\documentclass[10pt]{amsart}
\usepackage{amsmath}
\usepackage{cases}
\usepackage[colorlinks,
            linkcolor=black,
            anchorcolor=black,
            citecolor=black,
            urlcolor=black
            ]{hyperref}

 \newtheorem{thm}{Theorem}[section]

 \newtheorem{abcdef}{Theorem}[section]
 \newtheorem{lem}[thm]{Lemma}
 \newtheorem{prop}[thm]{Proposition}
 \newtheorem{abcd}{Theorem}[section]
 \theoremstyle{remark}
 \newtheorem{remark}[abcd]{Remark}
 \theoremstyle{definition}
 \newtheorem{definition}[abcdef]{Definition}
 \numberwithin{equation}{section}

\begin{document}

\title[Optimal $L^2$ extension from subvarieties]
{Optimal $L^2$ extension of sections from subvarieties in weakly
pseudoconvex manifolds}

\author{Xiangyu Zhou, Langfeng Zhu}

\address{Xiangyu Zhou: Institute of Mathematics, AMSS, and Hua Loo-Keng Key Laboratory of Mathematics,
Chinese Academy of Sciences, Beijing 100190, China}
\email{xyzhou@math.ac.cn}

\address{Langfeng Zhu: School of Mathematics and Statistics, Wuhan University, Wuhan 430072, China}
\email{zhulangfeng@amss.ac.cn}

\thanks{Zhou was partially supported by the National Natural Science Foundation of China. Zhu was partially supported by the National Natural Science Foundation of China
(No. 11201347 and No. 11671306) and the China Scholarship Council.}



\begin{abstract}
In this paper, we obtain optimal $L^2$ extension of holomorphic
sections of a holomorphic vector bundle from subvarieties in weakly
pseudoconvex K\"{a}hler manifolds. Moreover, in the case of line
bundle the Hermitian metric is allowed to be singular .
\end{abstract}

\maketitle


\section{Introduction and main results}\label{section-introduction}

The $L^2$ extension problem is an important topic in several complex variables and complex geometry. Many generalizations and applications have been obtained since the original work of Ohsawa and Takegoshi (\cite{Ohsawa-Takegoshi}). A recent progress is about the optimal $L^2$ extension and its applications.

Most recently, several general $L^2$ extension theorems with optimal estimates were proved in \cite{Guan-Zhou2013b} for holomorphic sections defined on subvarieties in Stein or projective manifolds. In \cite{Demailly2015}, several $L^2$ extension theorems were obtained for holomorphic sections defined on subvarieties in weakly pseudoconvex K\"{a}hler manifolds.

In this paper, we prove an optimal $L^2$ extension theorem, which generalizes the main theorems in \cite{Guan-Zhou2013b} to weakly pseudoconvex K\"{a}hler manifolds and optimizes a main theorem in \cite{Demailly2015} (cf. Theorem 2.8 and Remark 2.9 in \cite{Demailly2015}).

Let us recall some definitions in \cite{Demailly2015}.

\begin{definition}
A function $\psi:\,X\longrightarrow[-\infty,+\infty)$ on a complex manifold $X$ is said to be quasi-plurisubharmonic if $\psi$ is locally the sum of a plurisubharmonic function and a smooth function. In addition,
we say that $\psi$ has neat analytic singularities if every point $x\in X$ possesses an open
neighborhood $U$ on which $\psi$ can be written as
\[\psi = c\log\sum\limits_{1\leq j\leq j_0}|g_j|^2+u,\]
where $c$ is a nonnegative number, $g_j\in\mathcal{O}_X(U)$ and $u\in C^\infty(U)$.
\end{definition}

\begin{definition}
If $\psi$ is a quasi-plurisubharmonic function on a complex manifold $X$, the multiplier
ideal sheaf $\mathcal{I}(\psi)$ is the coherent analytic subsheaf of $\mathcal{O}_X$ defined by
\[\mathcal{I}(\psi)_x=\{f\in \mathcal{O}_{X,x}:\,\exists\, U\ni x,\,\int_U|f|^2e^{-\psi}d\lambda<+\infty\},\]
where $U$ is an open coordinate neighborhood of $x$, and $d\lambda$ is the Lebesgue measure in the corresponding open chart of $\mathbb{C}^n$. We say that the singularities of $\psi$ are log canonical
along the zero variety $Y=V(\mathcal{I}(\psi))$ if $\mathcal{I}((1-\varepsilon)\psi)\big|_Y=\mathcal{O}_X\big|_Y$ for every $\varepsilon>0$.
\end{definition}

If $\omega$ is a K\"{a}hler metric on $X$, we let $dV_{X,\omega}:=\frac{\omega^n}{n!}$ be the corresponding K\"{a}hler volume element, where $n=\dim X$. In case $\psi$ has log canonical singularities along $Y=V(\mathcal{I}(\psi))$, one can associate in a natural way a measure $dV_{X,\omega}[\psi]$ on the set $Y^0=Y_{\mathrm{reg}}$ of regular points of $Y$ as follows.

\begin{definition}\label{d:measure}
If $g \in C_c(Y^0)$ is a compactly supported nonnegative continuous function on $Y^0$ and $\widetilde{g}$ is a compactly supported nonnegative continuous extension of $g$ to $X$ such that $(\mathrm{supp}\,\widetilde{g})\cap Y\subset Y^0$, then we set
\[\int_{Y^0}g\,dV_{X,\omega}[\psi]=\varlimsup\limits_{t\rightarrow-\infty}\int_{\{x\in X:\,t<\psi(x)<t+1\}}\widetilde{g}e^{-\psi}dV_{X,\omega}.\]
\end{definition}

\begin{remark}
By Hironaka's desingularization theorem \ref{l:Hironaka}, it is not
hard to see that the limit in the above definition does not depend
on the extension $\widetilde{g}$ and then $dV_{X,\omega}[\psi]$ is
well defined on $Y^0$ (see Proposition 4.5 in \cite{Demailly2015}
for a proof).
\end{remark}

\begin{remark}\label{r:definition-of-measure}
The definition of $dV_{X,\omega}[\psi]$ here has a slight difference with the one in \cite{Guan-Zhou2013b}. In fact, if we denote the measure in \cite{Guan-Zhou2013b} by $d\widehat{V}_{X,\omega}[\psi]$, the integral $\int_{Y^0}g\,dV_{X,\omega}[\psi]$ here is equal to
\[\sum\limits_{1\leq j\leq n}\frac{\pi^j}{j!}\int_{Y_{n-j}}g\,d\widehat{V}_{X,\omega}[\psi],\]
where $Y_{n-j}$ is the $(n-j)$-dimensional component of $Y_{\mathrm{reg}}$.
\end{remark}

We will define a class of functions before the statement of our main theorem.

\begin{definition}
Let $\alpha_0\in(-\infty,+\infty]$ and $\alpha_1\in[0,+\infty)$.
When $\alpha_0\neq+\infty$, let $\mathfrak{R}_{\alpha_0,\alpha_1}$ be the class of functions defined by
\begin{align*}
\big\{R\in&
C^{\infty}(-\infty,\alpha_0]:\text{ }R>0,\text{ }R\text{ is decreasing near }-\infty,
\\&\varlimsup\limits_{t\rightarrow-\infty} e^tR(t)<+\infty,
\text{ }C_R:=\int_{-\infty}^{\alpha_0}\frac{1}{R(t)}dt<+\infty\text{ and}
\\&\int_t^{\alpha_0}\bigg(\frac{\alpha_1}{R(\alpha_0)}+\int_{t_2}^{\alpha_0}\frac{dt_1}{R(t_1)}\bigg)dt_2+\frac{(\alpha_1)^2}{R(\alpha_0)}
<R(t)\bigg(\frac{\alpha_1}{R(\alpha_0)}+\int_t^{\alpha_0}\frac{dt_1}{R(t_1)}\bigg)^2
\\&\text{for all }t\in(-\infty,\alpha_0)\big\}.
\end{align*}
When $\alpha_0=+\infty$, we replace $R\in C^\infty(-\infty,\alpha_0]$ with $R\in C^\infty(-\infty,+\infty)$ and $R(+\infty):=\varliminf\limits_{t\rightarrow+\infty} R(t)\in(0,+\infty]$ in the above definition of $\mathfrak{R}_{\alpha_0,\alpha_1}$.
\end{definition}

\begin{remark}\label{r:defintion-of-function-class}
The number $\alpha_0$, $\alpha_1$ and the function $R(t)$ are equal to the number $A$, $\frac{1}{\delta}$ and the function $\frac{1}{c_A(-t)e^t}$ defined just before the main theorems in \cite{Guan-Zhou2013b}. If $\alpha_0\neq+\infty$ and $R$ is decreasing on $(-\infty,\alpha_0]$, the longest inequality in the definition of $\mathfrak{R}_{\alpha_0,\alpha_1}$ holds for all $t\in(-\infty,\alpha_0)$. If $\alpha_0=+\infty$, the longest inequality in the definition of $\mathfrak{R}_{\alpha_0,\alpha_1}$ implies that
$\int_t^{+\infty}\frac{\alpha_1}{R(+\infty)}dt_2<+\infty$ for all $t\in(-\infty,+\infty)$. Therefore, $\frac{\alpha_1}{R(+\infty)}=0$, i.e., $\alpha_1=0$ or $R(+\infty)=+\infty$.
\end{remark}

\begin{thm}[The main theorem]\label{t:Zhou-Zhu1}
Let $R\in\mathfrak{R}_{\alpha_0,\alpha_1}$. Let
$(X,\omega)$ be a weakly pseudoconvex complex $n$-dimensional
manifold possessing a K\"{a}hler metric $\omega$, and $\psi$ be a quasi-plurisubharmonic function on $X$ with neat analytic singularities. Let
$Y$ be the analytic subvariety of $X$ defined by $Y=V(\mathcal{I}(\psi))$ and assume that $\psi$ has log canonical singularities along $Y$.
Let $L$ (resp. $E$) be a holomorphic line bundle (resp. a holomorphic vector bundle) over $X$ equipped with a
singular Hermitian metric $h=h_L$ (resp. a smooth Hermitian metric $h=h_E$), which is written locally as
$e^{-\phi_L}$ for some quasi-plurisubharmonic function $\phi_L$
with respect to a local holomorphic frame of $L$. Assume that
\begin{enumerate}
\item[$(i)$]
$\sqrt{-1}\Theta_h+\sqrt{-1}\partial\bar\partial\psi$ is semi-positive on $X\setminus\{\psi=-\infty\}$ in the sense of currents (resp. in the sense of Nakano),
\end{enumerate}
and that there is a continuous function
$\alpha<\alpha_0$ on $X$ such that the following two assumptions hold:
\begin{enumerate}
\item[$(ii)$]
$\sqrt{-1}\Theta_h+\sqrt{-1}\partial\bar\partial\psi
+\frac{1}{\widetilde{\chi}(\alpha)}\sqrt{-1}\partial\bar\partial\psi$ is semi-positive on $X\setminus\{\psi=-\infty\}$ in the sense of currents
(resp. in the sense of Nakano),
\item[$(iii)$]
$\psi\leq\alpha$,
\end{enumerate}
where $\widetilde{\chi}(t)$ is the function
\begin{equation}\label{e:definition of chi}
\frac{\int_t^{\alpha_0}\big(\frac{\alpha_1}{R(\alpha_0)}+\int_{t_2}^{\alpha_0}\frac{dt_1}{R(t_1)}\big)dt_2+\frac{(\alpha_1)^2}{R(\alpha_0)}}
{\frac{\alpha_1}{R(\alpha_0)}+\int_t^{\alpha_0}\frac{dt_1}{R(t_1)}}.
\end{equation}
Then for every section $f\in H^0\big(Y^0,(K_X\otimes L)\big|_{Y^0}\big)$ (resp. $f\in H^0\big(Y^0,(K_X\otimes E)\big|_{Y^0}\big)$) on $Y^0=Y_{\mathrm{reg}}$ such that
\begin{equation}\label{ie:thm-f-finite}
\int_{Y^0}|f|^2_{\omega,h} dV_{X,\omega}[\psi]<+\infty,
\end{equation}
there exists a section $F\in H^0(X,K_X\otimes L)$ (resp. $F\in H^0(X,K_X\otimes E)$) such that $F=f$ on $Y^0$ and
\begin{equation}\label{ie:final estimate}
\int_X\frac{|F|^2_{\omega,h}}{e^{\psi}R(\psi)} dV_{X,\omega}
\leq \bigg(\frac{\alpha_1}{R(\alpha_0)}+C_R\bigg)\int_{Y^0}|f|^2_{\omega,h} dV_{X,\omega}[\psi].
\end{equation}
\end{thm}

\begin{remark}\label{r:thm-Stein}
The case of Theorem \ref{t:Zhou-Zhu1} when $X$ is Stein or
projective was proved in \cite{Guan-Zhou2013b} (see also Proposition
4.1 in \cite{Zhou-Zhu2015} for a simplified version). Hence Theorem
\ref{t:Zhou-Zhu1} can be regarded as a generalization of the main
theorems in \cite{Guan-Zhou2013b} to weakly pseudoconvex K\"{a}hler
manifolds. Then it is easy to see from Remark
\ref{r:definition-of-measure} and the main theorems in
\cite{Guan-Zhou2013b} that the constant
$\frac{\alpha_1}{R(\alpha_0)}+C_R$ in $(\ref{ie:final estimate})$ is
optimal. Hence Theorem \ref{t:Zhou-Zhu1} gives an optimal version of
a main theorem in \cite{Demailly2015} (cf. Theorem 2.8 and Remark
2.9 in \cite{Demailly2015}).
\end{remark}

\begin{remark}
In \cite{Zhou-Zhu2015}, Theorem \ref{t:Zhou-Zhu1} was proved for $L$ in the special case when $\psi=m\log|s|^2$, $\alpha_0=\alpha_1=0$ and $R$ is decreasing on $(-\infty,0]$, where $s$ is a global holomorphic section of some holomorphic vector bundle of rank $m$ over $X$ equipped with a smooth Hermitian metric, and $s$ is transverse to the zero section. Similarly as in \cite{Zhou-Zhu2015}, a global plurisubharmonic negligible weight can be added to Theorem \ref{t:Zhou-Zhu1} by adding another regularization process to Step 2 in Section \ref{section-proof-of-theorem}.
\end{remark}

\begin{remark}
In order to deal with the singular metric $h_L$ on the weakly
pseudoconvex K\"{a}hler manifold $X$, not only the regularization
theorem \ref{l:Demailly1994} and the error term method of solving
$\bar\partial$ equations (Lemma \ref{l:Demailly-non complete
metric}) are needed, but also a limit problem about $L^2$ integrals
with singular weights needs to be solved. We solve the limit problem
in Proposition \ref{p:convergence-integral-psh}. Then by using
Proposition \ref{p:special-L2-extension}, Proposition
\ref{p:convergence-integral-psh} and the strong openness property of
multiplier ideal sheaves (Theorem \ref{l:strong-openness}) as the
key tools, we construct a family of smooth extensions of $f$
satisfying some uniform estimates, and overcome the difficulty in
dealing with the singular metric (see also \cite{Zhou-Zhu2015} for
the special case).
\end{remark}

The rest sections of this paper are organized as follows. First, we give some results used in the proof of Theorem \ref{t:Zhou-Zhu1} in Section \ref{section-lemmas}. Then, we prove two key propositions in Section \ref{section-singular-metric} which will be used to deal with the singular metric $h_L$. Finally, we prove Theorem \ref{t:Zhou-Zhu1} in Section \ref{section-proof-of-theorem} by using the results in Section \ref{section-lemmas} and Section \ref{section-singular-metric}.


\section{Some results used in the proof of
Theorem \ref{t:Zhou-Zhu1}}\label{section-lemmas}

In this section, we give some results which will be used in the
proof of Theorem \ref{t:Zhou-Zhu1}.

\begin{lem}[\cite{Demailly99}, \cite{Demailly2015}]\label{l:Demailly-non complete metric}
Let $( X ,\omega)$ be a complete K\"{a}hler manifold equipped with a
(non necessarily complete) K\"{a}hler metric $\omega$, and let $(Q,h)$
be a holomorphic vector bundle over $ X $ equipped with a smooth Hermitian metric $h$. Assume that $\tau$ and $A$
are smooth and bounded positive functions on $ X $ and let
\[\mathrm{B}:=[\tau\sqrt{-1}\Theta_{Q,h}-\sqrt{-1}
\partial\bar\partial\tau
-\sqrt{-1}A^{-1}\partial\tau\wedge\bar\partial\tau,\Lambda ].\]
Assume that $\delta\geq0$ is a nonnegative number such that $\mathrm{B}+\delta
\mathrm{I}$ is semi-positive definite everywhere on
$\wedge^{n,q}T^*_ X \otimes Q$ for some $q\geq 1$. Then given a form
$g\in L^2( X ,\wedge^{n,q}T^*_ X \otimes Q)$ such that $\mathrm{D}''
g=0$ and
\[\int_ X \langle {(\mathrm{B}+\delta
\mathrm{I})}^{-1}g,g\rangle_{\omega,h} dV_{X,\omega}<+\infty,\]
there exists an
approximate solution $u\in L^2( X ,\wedge^{n,q-1}T^*_ X \otimes Q)$
and a correcting term $v\in L^2( X ,\wedge^{n,q}T^*_ X \otimes Q)$
such that $\mathrm{D}'' u+\sqrt{\delta}v=g$ and
\[\int_ X \frac{|u|^2_{\omega,h}}{\tau+A} dV_{X,\omega}+\int_X|v|^2_{\omega,h} dV_{X,\omega}\leq \int_ X
\langle {(\mathrm{B}+\delta \mathrm{I})}^{-1}g,g\rangle_{\omega,h} dV_{X,\omega}.\]
\end{lem}

\begin{thm}[Theorem 6.1 in \cite{Demailly94}]\label{l:Demailly1994}
Let $(X,\omega)$ be a complex manifold equipped with a Hermitian
metric $\omega$, and $\Omega\subset\subset X$ be an open subset.
Assume that $T=\widetilde{T}+\frac{\sqrt{-1}}{\pi}
\partial\bar\partial\varphi$ is a closed $(1,1)$-current on $X$,
where $\widetilde{T}$ is a smooth real $(1,1)$-form and $\varphi$ is
a quasi-plurisubharmonic function. Let $\gamma$ be a continuous real
$(1,1)$-form such that $T\geq \gamma$. Suppose that the Chern
curvature tensor of $T_X$ satisfies
\[(\sqrt{-1}\Theta_{T_X}+\varpi\otimes\mathrm{Id}_{T_X})
(\kappa_1\otimes\kappa_2,\kappa_1\otimes\kappa_2)
\geq0\quad(\forall\kappa_1,\kappa_2\in T_X\text{ with }\langle
\kappa_1,\kappa_2\rangle=0)\]
for some continuous nonnegative
$(1,1)$-form $\varpi$ on $X$. Then there is a family of closed
$(1,1)$-currents $T_{\varsigma,\rho}=\widetilde{T}
+\frac{\sqrt{-1}}{\pi}\partial\bar\partial \varphi_{\varsigma,\rho}$ defined
on a neighborhood of $\overline{\Omega}$ ($\varsigma\in(0,+\infty)$ and $\rho\in(0,\rho_1)$ for some
positive number $\rho_1$) independent of $\gamma$, such that
\begin{enumerate}
\item[$(i)$]
$\varphi_{\varsigma,\rho}$ is quasi-plurisubharmonic on a neighborhood of $\overline{\Omega}$,
smooth on $\Omega\setminus E_\varsigma(T)$, increasing with respect to
$\varsigma$ and $\rho$ on $\Omega$, and converges to $\varphi$ on
$\Omega$ as $\rho\rightarrow0$,
\item[$(ii)$]
$T_{\varsigma,\rho}\geq\gamma-\varsigma \varpi-\delta_\rho\omega$ on $\Omega$,
\end{enumerate}
where $E_\varsigma(T):=\{x\in X:\,\nu(T,x)\geq \varsigma\}$
$(\varsigma>0)$ is the $\varsigma$-upperlevel set of Lelong numbers,
and $\{\delta_\rho\}$ is an increasing family of positive numbers
such that $\lim\limits_{\rho\rightarrow 0}\delta_\rho=0$.
\end{thm}

\begin{remark}
Although Lemma \ref{l:Demailly1994} is stated in \cite{Demailly94} in the case $X$ is compact,
almost the same proof as in \cite{Demailly94} shows that Lemma
\ref{l:Demailly1994} holds in the noncompact case while uniform
estimates are obtained only on the relatively compact subset
$\Omega$.
\end{remark}

\begin{lem}[Theorem 1.5 in \cite{Demailly82}]\label{l:Demailly1982-complete-metric}
Let $X$ be a K\"{a}hler
manifold, and $Z$ be an analytic subset of $X$. Assume that $\Omega$
is a relatively compact open subset of $X$ possessing a complete
K\"{a}hler metric. Then $\Omega\setminus Z$ carries a complete
K\"{a}hler metric.
\end{lem}

\begin{lem}[Theorem 4.4.2 in \cite{Hormander}]\label{l:Hormander-dbar}
Let $\Omega$ be a pseudoconvex
open set in $\mathbb{C}^n$, and $\varphi$ be a plurisubharmonic
function on $\Omega$. For every $w\in L^2_{(p,q+1)}(\Omega,e^{-\varphi})$
with $\bar\partial w=0$ there is a solution $s\in
L^2_{(p,q)}(\Omega,\mathrm{loc})$ of the equation $\bar\partial s=w$
such that
\[\int_\Omega\frac{|s|^2}{(1+|z|^2)^2}e^{-\varphi}d\lambda\leq
\int_\Omega|w|^2e^{-\varphi}d\lambda,\]
where $d\lambda$ is the $2n$-dimensional Lebesgue measure on
$\mathbb{C}^n$.
\end{lem}

\begin{lem}[Lemma 6.9 in \cite{Demailly82}]\label{l:extension}
Let $\Omega$ be an open subset of
$\mathbb{C}^n$ and $Z$ be a complex analytic subset of $\Omega$.
Assume that $u$ is a $(p,q-1)$-form with $L^2_{\mathrm{loc}}$
coefficients and $g$ is a $(p,q)$-form with $L^1_{\mathrm{loc}}$
coefficients such that $\bar\partial u=g$ on $\Omega\setminus Z$ (in the sense of currents). Then $\bar\partial u=g$ on
$\Omega$.
\end{lem}

\begin{thm}[Strong openness property of multiplier ideal sheaves, \cite{Guan-Zhou2013c}]\label{l:strong-openness}
Let $\varphi$ be a negative plurisubharmonic function on the unit
polydisk $\Delta^n\subset\mathbb{C}^n$. Assume that $F$ is a
holomorphic function on $\Delta^n$ satisfying
\[\int_{\Delta^n}|F|^2e^{-\varphi}d\lambda<+\infty.\]
Then there exists $r\in(0,1)$ and $\beta\in
(0,+\infty)$ such that
\[\int_{\Delta^n_r}|F|^2e^{-(1+\beta)\varphi}d\lambda<+\infty,\]
where $\Delta^n_r:=\{(z_1,\cdots,z_n)\in
\mathbb{C}^n:\,|z_k|<r,1\leq k\leq n\}$.
\end{thm}

\begin{thm}[Hironaka's desingularization theorem, \cite{Hironaka1964}, \cite{Bierstone-Milman1991}]\label{l:Hironaka}
Let $X$ be a complex manifold, and $M$ be an analytic subvariety in $X$. Then there is a local finite sequence of blow-ups $\mu_j:X_{j+1}\longrightarrow X_j$ $(X_1:=X,\,j=1,2,\cdots)$ with smooth centers $S_j$ such that:
\begin{enumerate}
\item[$(1)$]
Each component of $S_j$ lies either in $(M_j)_{\mathrm{sing}}$ or in $M_j\cap E_j$, where $M_1:=M$, $M_{j+1}$ denotes the strict transform of $M_j$ by $\mu_j$, $(M_j)_{\mathrm{sing}}$ denotes the singular set of $M_j$, and $E_{j+1}$ denotes the exceptional divisor $\mu_j^{-1}(S_j\cup E_j)$.
\item[$(2)$]
Let $M'$ and $E'$ denote the final strict transform of $M$ and the exceptional divisor respectively. Then:
\begin{enumerate}
\item[$(i)$]
The underlying point-set $|M'|$ is smooth.
\item[$(ii)$]
$|M'|$ and $E'$ simultaneously have only normal crossings.
\end{enumerate}
\end{enumerate}
\end{thm}

\begin{remark}
We say that $|M'|$ and $E'$ simultaneously have only normal crossings if, locally, there is a coordinate system in which $E'$ is a union of coordinate hyperplanes, and $|M'|$ is a coordinate subspace.
\end{remark}


\section{Key propositions used to deal with the singular metric $h_L$}\label{section-singular-metric}

In order to deal with the singular metric $h_L$, we will prove two key propositions in this section, which are generalizations of the key propositions in \cite{Zhou-Zhu2015}.

\begin{prop}\label{p:special-L2-extension}
Let $R$ be a positive continuous function defined on $(-\infty,0]$ such that $\beta_R:=\sup\limits_{t\leq0}\big(e^tR(t)\big)<+\infty$ and $\widehat{\beta}_R:=\inf\limits_{t\leq0}R(t)>0$. Let
$\Omega\subset\mathbb{C}^n$ be a bounded pseudoconvex domain, $\phi$ be a
plurisubharmonic function on $\Omega$, and $\Upsilon$ be a
quasi-plurisubharmonic function defined on a neighborhood on $\overline{\Omega}$. Assume that $\Upsilon$ has neat analytic singularities and the singularities of $\Upsilon$ are log canonical along the zero variety $Y=V(\mathcal{I}(\Upsilon))$. Set
\[U=\{x\in\Omega:\,\Upsilon(x)<0\}.\]
Furthermore, assume that
\[\sqrt{-1}\partial\bar\partial\Upsilon\geq-\gamma\sqrt{-1}\partial\bar\partial|z|^2\]
on $\Omega$ for some nonnegative number $\gamma$, where $z:=(z_1,\cdots,z_n)$ is the coordinate vector in $\mathbb{C}^n$.
Then for every $\beta_1\in(0,1)$ and every
holomorphic $n$-form $f$ on $U$ satisfying
\[\int_U\frac{|f|^2e^{-\phi}}{e^{\Upsilon}R(\Upsilon )}d\lambda<+\infty,\]
there exists a holomorphic $n$-form $F$ on $\Omega$ satisfying $F=f$
on $Y$,
\begin{equation}\label{ie:special-extension1}
\int_U\frac{|F|^2e^{-\phi}d\lambda}{ e^{\Upsilon}R(\Upsilon )}\leq
e^{2\gamma\sup\limits_{\Omega}|z|^2}\bigg(2+\frac{72\beta_R}{\beta_1\widehat{\beta}_R}\bigg)
\int_U\frac{|f|^2e^{-\phi}d\lambda}{ e^{\Upsilon}R(\Upsilon )},
\end{equation}
and
\begin{equation}\label{ie:special-extension2}
\int_\Omega\frac{|F|^2e^{-\phi}d\lambda}{(1+e^\Upsilon)^{1+\beta_1}}
\leq e^{2\gamma\sup\limits_{\Omega}|z|^2}\bigg(\beta_R+\frac{36\beta_R}{\beta_12^{\beta_1}}\bigg)
\int_U\frac{|f|^2e^{-\phi}d\lambda}{ e^{\Upsilon}R(\Upsilon )}.
\end{equation}
\end{prop}

\begin{proof}
This proposition is a modification of a theorem in \cite{Demailly-book}.

Since $\Omega$ is a pseudoconvex domain, there is a sequence of
pseudoconvex subdomains $\Omega_k\subset\subset\Omega$
$(k=1,2,\cdots)$ such that
$\overset{+\infty}{\underset{k=1}\cup}\Omega_k=\Omega$. Then for
fixed $k$, by convolution we can get a decreasing family of smooth
plurisubharmonic functions $\{\phi_j\}_{j=1}^{+\infty}$ defined on a
neighborhood of $\overline{\Omega_k}$ such that
$\lim\limits_{j\rightarrow+\infty}\phi_j=\phi$.

Let $\theta:\mathbb{R}\longrightarrow[0,1]$ be a smooth function such that $\theta=1$ on $(-\infty,\frac{1}{4})$, $\theta=0$ on $(\frac{3}{4},+\infty)$ and $|\theta'|\leq3$ on $\mathbb{R}$.

Fix $k$ and $j$. Set $\widehat{f}=\theta(e^\Upsilon)f$.
Then the construction of $\widehat{f}$ implies that $\widehat{f}$ is smooth on $\Omega$ and $\widehat{f}=f$ on $Y\cap\Omega$.

Set
$g=\bar\partial \widehat{f}$. Then $g=\theta'(e^\Upsilon)e^\Upsilon\bar\partial\Upsilon\wedge f$ on $\Omega$.

Let $\Sigma:=\{\Upsilon=-\infty\}$. Lemma \ref{l:Demailly1982-complete-metric} implies that
$\Omega_k\setminus\Sigma$ is a complete K\"{a}hler manifold. Let
$\Omega_k\setminus\Sigma$ be endowed with the Euclidean metric and let
$Q$ be the trivial line bundle on $\Omega_k\setminus\Sigma$ equipped
with the metric
\[h:=e^{-\phi_j-\Upsilon -\beta_1\log(1+e^\Upsilon)-2\gamma|z|^2}.\]
Then we want to solve a $\bar\partial$ equation on $\Omega_k\setminus\Sigma$ by
applying Lemma \ref{l:Demailly-non complete metric} to the case
$\tau=1$, $A=0$ and $\delta=0$ (in fact, the case $\tau=1$ and $A=0$
is the non twisted version of Lemma \ref{l:Demailly-non complete metric}). The key step in applying Lemma \ref{l:Demailly-non complete metric} is to estimate the term
\[\int_{\Omega_k\setminus\Sigma}\langle \mathrm{B}^{-1}g,g\rangle_h d\lambda,\]
where $\mathrm{B}:=[\sqrt{-1}\Theta_h ,\Lambda]$.

Set $\nu=\partial\Upsilon$. Then
$g=\theta'(e^\Upsilon)e^\Upsilon\bar\nu\wedge f$ on $\Omega$.

Since
\begin{eqnarray*}
& &\sqrt{-1}\Theta_h \big|_{\Omega_k\setminus\Sigma}\\
&=&\sqrt{-1}\partial\bar\partial\phi_j+ \sqrt{-1}\partial\bar\partial\Upsilon
+\beta_1\sqrt{-1}\partial\bar\partial\log(1+e^\Upsilon)+2\gamma\sqrt{-1}\partial\bar\partial|z|^2\\
&=&\sqrt{-1}\partial\bar\partial\phi_j+\bigg(1+\frac{\beta_1e^\Upsilon}{1+e^\Upsilon}\bigg)
\sqrt{-1}\partial\bar\partial\Upsilon+2\gamma\sqrt{-1}\partial\bar\partial|z|^2
+\frac{\beta_1e^\Upsilon\sqrt{-1}\partial\Upsilon\wedge\bar\partial\Upsilon}{(1+e^\Upsilon)^2}\\
&\geq&\frac{\beta_1e^\Upsilon\sqrt{-1}\nu\wedge\bar \nu}{(1+e^\Upsilon)^2},
\end{eqnarray*}
we get
\[\mathrm{B}\geq\frac{\beta_1e^\Upsilon}{(1+e^\Upsilon)^2}\mathrm{T}_{\bar\nu}\mathrm{T}^*_{\bar\nu}\]
on $\Omega_k\setminus\Sigma$, where $T_{\bar\nu}$ denotes the operator $\bar\nu\wedge\bullet$ and $\mathrm{T}^*_{\bar\nu}$ is its Hilbert adjoint operator. Then we get $\langle
\mathrm{B}^{-1}g,g\rangle_h\big|_{\Omega_k\setminus U}=0$ and
\begin{eqnarray*}
& &\langle\mathrm{B}^{-1}g,g\rangle_h\big|_{(U\cap\Omega_k)\setminus\Sigma}\\
&=& \langle\mathrm{B}^{-1} (\theta'(e^\Upsilon)e^\Upsilon\bar\nu\wedge f ),
\theta'(e^\Upsilon)e^\Upsilon\bar\nu\wedge f \rangle_h\\
&\leq&\frac{(1+e^\Upsilon)^2}{\beta_1e^\Upsilon} |\theta'(e^\Upsilon)e^\Upsilon f
 |^2e^{-\phi_j-\Upsilon -\beta_1\log(1+e^\Upsilon)-2\gamma|z|^2}\\
&=&\frac{(1+e^\Upsilon)^{2-\beta_1}}{\beta_1} |\theta'(e^\Upsilon)f
 |^2e^{-\phi_j-2\gamma|z|^2}\\
&\leq&\frac{36}{\beta_12^{\beta_1}} |f
 |^2e^{-\phi_j-2\gamma|z|^2}.
\end{eqnarray*}
Hence it follows from Lemma \ref{l:Demailly-non complete metric}
that there exists $u_{k,j}\in L^2(\Omega_k\setminus\Sigma,\,K_\Omega\otimes Q,\,h)$ such that $\bar\partial u_{k,j}=g=\bar\partial
\widehat{f}$ on $\Omega_k\setminus\Sigma$ and
\[\int_{\Omega_k\setminus\Sigma}|u_{k,j}|^2_h d\lambda
\leq\int_{\Omega_k\setminus\Sigma}\langle \mathrm{B}^{-1}g,g\rangle_h
d\lambda.\]
Thus
\begin{eqnarray}
& &\int_{\Omega_k\setminus\Sigma}\frac{|u_{k,j}|^2e^{-\phi_j-2\gamma|z|^2}}{ e^{\Upsilon}(1+e^{\Upsilon})^{\beta_1}}d\lambda\label{ie:prop1-u-and-f}\\
&\leq&\frac{36}{\beta_12^{\beta_1}}\int_{U\cap\Omega_k}|f|^2e^{-\phi_j-2\gamma|z|^2}d\lambda\nonumber\\
&\leq&\frac{36\beta_R}{\beta_12^{\beta_1}}
\int_U\frac{|f|^2e^{-\phi-2\gamma|z|^2}}{ e^{\Upsilon}R(\Upsilon )}d\lambda.\nonumber
\end{eqnarray}
Hence we have $u_{k,j}\in L^2(\Omega_k\setminus\Sigma,\,K_\Omega)$. Since
$g\in C^\infty(\Omega_k,\,\wedge^{n,1}T^*_\Omega)$, Lemma
\ref{l:extension} implies that $\bar\partial u_{k,j}=g$ holds on
$\Omega_k$.

Let $F_{k,j}:=\widehat{f}-u_{k,j}$. Then $\bar\partial F_{k,j}=0$ on
$\Omega_k$. Thus $F_{k,j}$ is holomorphic on $\Omega_k$. Hence
$u_{k,j}$ is smooth on $\Omega_k$. Then the non-integrability of
$e^{-\Upsilon}$ along $Y$ implies that $u_{k,j}=0$ on $Y\cap \Omega_k$.
Therefore, $F_{k,j}=f$ on $Y\cap\Omega_k$.

It follows from $(\ref{ie:prop1-u-and-f})$ that
\begin{eqnarray*}
&
&\int_{U\cap\Omega_k}\frac{|u_{k,j}|^2e^{-\phi_j-2\gamma|z|^2}}{ e^{\Upsilon}R(\Upsilon )}d\lambda\\
&\leq&\frac{2^{\beta_1}}{\widehat{\beta}_R}\int_{U\cap\Omega_k}
\frac{|u_{k,j}|^2e^{-\phi_j-2\gamma|z|^2}}{ e^{\Upsilon}(1+e^{\Upsilon})^{\beta_1}}d\lambda\\
&\leq&\frac{36\beta_R}{\beta_1\widehat{\beta}_R}
\int_U\frac{|f|^2e^{-\phi-2\gamma|z|^2}}{ e^{\Upsilon}R(\Upsilon )}d\lambda.
\end{eqnarray*}

Since
\[|F_{k,j}|^2\big|_{U\cap\Omega_k}\leq2|\widehat{f}|^2+2|u_{k,j}|^2
\leq2|f|^2+2|u_{k,j}|^2,\] we get
\begin{eqnarray}
&
&\int_{U\cap\Omega_k}\frac{|F_{k,j}|^2e^{-\phi_j-2\gamma|z|^2}}{ e^{\Upsilon}R(\Upsilon )}
d\lambda\label{ie:special-extension3}\\
&\leq&2\int_{U\cap\Omega_k}\frac{(|f|^2+|u_{k,j}|^2)
e^{-\phi_j-2\gamma|z|^2}}{ e^{\Upsilon}R(\Upsilon )}d\lambda\nonumber\\
&\leq&\bigg(2+\frac{72\beta_R}{\beta_1\widehat{\beta}_R}\bigg)
\int_U\frac{|f|^2e^{-\phi-2\gamma|z|^2}}{ e^{\Upsilon}R(\Upsilon )}d\lambda. \nonumber
\end{eqnarray}

Since
\begin{equation}\label{ie:inner product}
\langle\kappa_1+\kappa_2,\,\kappa_1+\kappa_2\rangle
\leq\langle\kappa_1,\kappa_1\rangle+\langle\kappa_2,
\kappa_2\rangle+c\langle\kappa_1,\kappa_1\rangle
+\frac{1}{c}\langle\kappa_2,\kappa_2\rangle
\end{equation}
for any inner product space $\big(\mathrm{H},
\,\langle\bullet,\bullet\rangle\big)$, where
$\kappa_1,\kappa_2\in\mathrm{H}$,
we get
\[|F_{k,j}|^2\big|_{U\cap\Omega_k}\leq(|f|+|u_{k,j}|)^2\leq(1+ e^{\Upsilon})|f|^2
+(1+\frac{1}{ e^{\Upsilon}})|u_{k,j}|^2.\]
Then
\[\frac{|F_{k,j}|^2}{(1+e^\Upsilon)^{1+\beta_1}}\bigg|_{U\cap\Omega_k}
\leq|f|^2+\frac{|u_{k,j}|^2}{ e^{\Upsilon}(1+e^\Upsilon)^{\beta_1}}.\]
Since $|F_{k,j}|^2\big|_{\Omega_k\setminus U}=|u_{k,j}|^2$,
we get
\[\frac{|F_{k,j}|^2}{(1+e^\Upsilon)^{1+\beta_1}}\bigg|_{\Omega_k\setminus U}
\leq\frac{|u_{k,j}|^2}{ e^{\Upsilon}(1+e^\Upsilon)^{\beta_1}}.\]
Hence it follows from the two inequalities above and $(\ref{ie:prop1-u-and-f})$ that
\begin{eqnarray}
&
&\int_{\Omega_k}\frac{|F_{k,j}|^2e^{-\phi_j-2\gamma|z|^2}}{(1+e^\Upsilon)^{1+\beta_1}}d\lambda\label{ie:special-extension4}\\
&\leq&\int_U|f|^2e^{-\phi-2\gamma|z|^2}d\lambda+\int_{\Omega_k}
\frac{|u_{k,j}|^2e^{-\phi_j-2\gamma|z|^2}}{ e^{\Upsilon}(1+e^\Upsilon)^{\beta_1}}d\lambda\nonumber\\
&\leq&\bigg(\beta_R+\frac{36\beta_R}{\beta_12^{\beta_1}}\bigg)
\int_U\frac{|f|^2e^{-\phi-2\gamma|z|^2}}{ e^{\Upsilon}R(\Upsilon )}d\lambda.\nonumber
\end{eqnarray}

Since $e^{-2\gamma\sup\limits_{\Omega}|z|^2}\leq e^{-2\gamma|z|^2}\leq1$ on $\Omega$, the desired holomorphic $n$-form $F$ on $\Omega$ and the $L^2$ estimates $(\ref{ie:special-extension1})$ and
$(\ref{ie:special-extension2})$ can be obtained from
$(\ref{ie:special-extension3})$ and $(\ref{ie:special-extension4})$
by applying Montel's theorem and extracting weak limits of
$\{F_{k,j}\}_{k,j}$, first as $j\rightarrow+\infty$ and then as
$k\rightarrow+\infty$.

\end{proof}

\begin{prop}\label{p:convergence-integral-psh}
Let $X$, $\psi$, $Y$ and $Y^0$ be as in Theorem \ref{t:Zhou-Zhu1}. Let $U\subset\subset V\subset\subset\Omega$ be three local coordinate balls in $X$, $\phi$ be a plurisubharmonic function on
$\Omega$ such that $\sup\limits_{\Omega}\phi<+\infty$, and $v$ be a nonnegative
continuous function on $\Omega$ with $\mathrm{supp}\,v\subset U$.
Let $C$, $\beta$, $c_1$ and $c_2$ be positive numbers, and let $\beta_1$ be a small enough positive number. Assume that $f$ is a holomorphic
function on $\Omega\cap Y$ satisfying
\begin{equation}\label{ie:prop2-f-L2}
\int_{\Omega\cap Y^0 }|f|^2e^{-\phi} d\lambda[\psi]<+\infty,
\end{equation}
and that $f_t\in\mathcal{O}(\Omega)$ $\big(t\in(-\infty,0)\big)$ are a family of holomorphic functions such that for all $t\in (-\infty,0)$,
$f_t=f$ on $\Omega\cap Y$,
\begin{equation}\label{ie:special-extension-1}
\sup\limits_{V}|f_t|^2 \leq
Ce^{-\beta_1t}
\end{equation}
and
\begin{equation}\label{ie:special-extension-2}
\frac{1}{e^t}\int_{\Omega\cap\{\psi<t+c_2\}}|f_t|^2e^{-(1+\beta)\phi}d\lambda\leq
C.
\end{equation}
Then
\begin{equation}\label{ie:prop2-final-estimate}
\varlimsup\limits_{ t\rightarrow-\infty }\int_{U\cap\{t-c_1<\psi<t+c_2\}}\frac{e^tv|f_t|^2e^{-\phi}} {(e^\psi+e^t)^2}d\lambda\leq\int_{U\cap Y^0 }v|f|^2e^{-\phi} d\lambda[\psi].
\end{equation}
\end{prop}

\begin{remark}
One of the key points in the proof of Proposition \ref{p:convergence-integral-psh} is to verify that the upper limit in $(\ref{ie:prop2-final-estimate})$ produces the zero measure on the singular set of $Y$, i.e., we have $(\ref{e:f-vanish})$. Then the key uniform estimates in Step 2 of the proof are obtained.
\end{remark}

In order to prove Proposition \ref{p:convergence-integral-psh}, we prove the following lemma at first.

\begin{lem}\label{l:one dim lemma}
Let $r_1$, $r_2$ and $\gamma$ be positive numbers such that $r_1<r_2<\gamma$. Let $\varphi$ be a bounded negative subharmonic function on $\Delta_\gamma$, where $\Delta_\gamma:=\{w\in\mathbb{C}:\,|w|<\gamma\}$. Assume that $\{v_t\}_{t\in(-\infty,0)}$ are nonnegative continuous functions defined on $\Delta_\gamma$ such that
\begin{equation}\label{e:uniform continuous}
\lim\limits_{t\rightarrow-\infty}\sup\limits_{\{w\in\mathbb{C}:\,e^t(r_1)^{2\alpha}<|w|^{2\alpha}<e^t(r_2)^{2\alpha}\}}|v_t(w)-v_0|=0,
\end{equation}
where $\alpha\in[1,+\infty)$ and $v_0\in[0,+\infty)$. Let
\[P_t:=\int_{\{w\in\mathbb{C}:\,e^t(r_1)^{2\alpha}<|w|^{2\alpha}<e^t(r_2)^{2\alpha}\}}
\frac{e^t|w|^{2\alpha-2}v_t(w)e^{-\varphi(w)}}{(|w|^{2\alpha}+e^t)^2}d\lambda(w).\]
Then
\begin{equation}\label{ie:one dim limit}
\varlimsup\limits_{t\rightarrow-\infty}P_t\leq\frac{\pi v_0e^{-\varphi(0)}}{\alpha}.
\end{equation}
\end{lem}

\begin{proof}
Put
\[S_{\delta,t}=\{z\in\Delta_\gamma:\,\varphi(e^{\frac{t}{2\alpha}}z)<(1+\delta)\varphi(0)\},\quad \delta\in(0,+\infty),\quad t\in(-\infty,0).\]
Denote by $\lambda(S_{\delta,t})$ the $2$-dimensional Lebesgue measure of $S_{\delta,t}$.

Since $\varphi(w)$ is a negative upper semicontinuous function on $\Delta_\gamma$ and $\varphi(0)>-\infty$, we have that for every $\varepsilon\in(0,1)$, there exists $t_\varepsilon\in(-\infty,0)$ such that
\[\varphi(e^{\frac{t}{2\alpha}}z)\leq(1-\varepsilon)\varphi(0)\]
for all $z\in\Delta_\gamma$ when $t\in(-\infty,t_\varepsilon)$.

Since $\varphi(e^{\frac{t}{2\alpha}}z)$ is subharmonic on $\Delta_\gamma$ with respect to $z$ for any $t\in(-\infty,t_\varepsilon)$, it follows from the mean value inequality that, for all $t\in(-\infty,t_\varepsilon)$,
\begin{eqnarray*}
\varphi(0)&\leq&\frac{1}{\pi\gamma^2}\int_{z\in\Delta_\gamma}\varphi(e^{\frac{t}{2\alpha}}z)d\lambda(z)\\
&=&\frac{1}{\pi\gamma^2}\int_{z\in\Delta_\gamma\setminus S_{\delta,t}}\varphi(e^{\frac{t}{2\alpha}}z)d\lambda(z)+\frac{1}{\pi\gamma^2}\int_{z\in S_{\delta,t}}\varphi(e^{\frac{t}{2\alpha}}z)d\lambda(z)\\
&\leq&\frac{(1-\varepsilon)\varphi(0)\big(\pi\gamma^2-\lambda(S_{\delta,t})\big)}{\pi\gamma^2}
+\frac{(1+\delta)\varphi(0)\lambda(S_{\delta,t})}{\pi\gamma^2}\\
&=&\varphi(0)\bigg(1-\varepsilon+\frac{(\delta+\varepsilon)\lambda(S_{\delta,t})}{\pi\gamma^2}\bigg).
\end{eqnarray*}
Then $\varphi(0)<0$ implies that
\[\lambda(S_{\delta,t})\leq\frac{\pi\gamma^2\varepsilon}{\delta+\varepsilon}\leq\frac{\pi\gamma^2}{\delta}\varepsilon\]
when $t\in(-\infty,t_\varepsilon)$. Hence
\begin{equation}\label{e:S limit 0}
\lim\limits_{t\rightarrow-\infty}\lambda(S_{\delta,t})=0,\quad \forall\,\delta\in(0,+\infty).
\end{equation}

Since $\varphi$ is bounded, we have
\[-\varphi\leq C_1\]
for some positive number $C_1$.

$(\ref{e:uniform continuous})$ implies that
\[\sup\limits_{\{w\in\mathbb{C}:\,e^t(r_1)^{2\alpha}<|w|^{2\alpha}<e^t(r_2)^{2\alpha}\}}v_t(w)\leq C_2\]
for some positive number $C_2$ independent of $t$ when $t$ is small enough.

Then by the change of variables $w=e^{\frac{t}{2\alpha}}z$, we have
\begin{eqnarray*}
P_t&=&\int_{\{z\in\mathbb{C}:\,r_1<|z|<r_2\}}
\frac{|z|^{2\alpha-2}v_t(e^{\frac{t}{2\alpha}}z)e^{-\varphi(e^{\frac{t}{2\alpha}}z)}}{(|z|^{2\alpha}+1)^2}d\lambda(z)\\
&=&\int_{\{r_1<|z|<r_2\}\cap S_{\delta,t}}
\frac{|z|^{2\alpha-2}v_t(e^{\frac{t}{2\alpha}}z)e^{-\varphi(e^{\frac{t}{2\alpha}}z)}}{(|z|^{2\alpha}+1)^2}d\lambda(z)\\
& &+\int_{\{r_1<|z|<r_2\}\setminus S_{\delta,t}}
\frac{|z|^{2\alpha-2}v_t(e^{\frac{t}{2\alpha}}z)e^{-\varphi(e^{\frac{t}{2\alpha}}z)}}{(|z|^{2\alpha}+1)^2}d\lambda(z)\\
&\leq&\frac{(r_2)^{2\alpha-2}C_2e^{C_1}}{\big((r_1)^{2\alpha}+1\big)^2}\cdot\lambda(S_{\delta,t})\\
& &+\bigg(\sup\limits_{r_1<|z|<r_2}v_t(e^{\frac{t}{2\alpha}}z)\bigg)e^{-(1+\delta)\varphi(0)}\int_{\{r_1<|z|<r_2\}}
\frac{|z|^{2\alpha-2}}{(|z|^{2\alpha}+1)^2}d\lambda(z).
\end{eqnarray*}
Since
\[\int_{\{r_1<|z|<r_2\}}
\frac{|z|^{2\alpha-2}}{(|z|^{2\alpha}+1)^2}d\lambda(z)\leq\frac{\pi}{\alpha},\]
we obtain from $(\ref{e:uniform continuous})$, $(\ref{e:S limit 0})$ that
\[\varlimsup\limits_{t\rightarrow-\infty}P_t\leq\frac{\pi v_0e^{-(1+\delta)\varphi(0)}}{\alpha}.\]
Since $\delta$ is an arbitrary positive number, we get $(\ref{ie:one dim limit})$.

\end{proof}

Now we begin to prove Proposition \ref{p:convergence-integral-psh}.

\begin{proof}
Let $\beta_v:=\sup\limits_{U}v$.

Without loss of generality, we may suppose that $\phi$ is negative
on $\Omega$.

We will use Hironaka's desingularization theorem (Lemma \ref{l:Hironaka}) to deal with the measure $d\lambda[\psi]$. This idea comes from the work \cite{Demailly2015}.

At first we use Lemma \ref{l:Hironaka} on $X$ to resolve the singularities of $Y$ and we denote the corresponding proper modification by $\mu_1$. Next, we make a blow-up $\mu_2$ along $|Y'|$. Then we use Lemma \ref{l:Hironaka} again to resolve the singularities of $\Sigma$ and we denote the corresponding proper holomorphic modification by $\mu_3$, where $\Sigma$ denote the strict transform of $\{\psi=-\infty\}$ by $\mu_1\circ\mu_2$. Finally, we make a blow-up $\mu_4$ along $|\Sigma'|$. Thus we can get a proper holomorphic map $\mu:\widetilde{X}\longrightarrow X$, which is locally a finite composition of blow-ups with smooth centers and is equal to $\mu_1\circ\mu_2\circ\mu_3\circ\mu_4$. Moreover, $\widetilde{Y}$ and the divisor $\overline{\mu^{-1}(\{\psi=-\infty\})\setminus \widetilde{Y}}$ simultaneously have only normal crossings in $\widetilde{X}$, where $\widetilde{Y}$ denotes the strict transform of $\mu_2^{-1}(|Y'|)$ by $\mu_3\circ\mu_4$.

\textbf{Step 1: we will represent the measure $|f|^2d\lambda[\psi]$ on $Y^0\cap U$ explicitly as an integral on $\widetilde{Y}$ (see $(\ref{e:prop2-limit1})$).}

For any $\widetilde{x}\in\overline{\mu^{-1}(U)}\cap\mu^{-1}(\{\psi=-\infty\})$, there exists a relatively compact coordinate ball $(W;w_1,\cdots,w_n)$ contained in $\mu^{-1}(V)$ centered at $\widetilde{x}$ such that
$w^b=0$ is the zero divisor of the Jacobian $J_\mu$,
and $\psi\circ\mu$ can be written on $W$ as
\[\psi\circ\mu(w)=c\log|w^a|^2+\widetilde u(w),\]
where $c$ is a positive number, $w:=(w_1,\cdots,w_n)$, $\widetilde u\in C^\infty(\overline{W})$, $w^a:=\prod\limits_{p=1}^n w_p^{a_p}$ and $w^b:=\prod\limits_{p=1}^nw_p^{b_p}$ for some nonnegative integers $a_p$ and $b_p$.

Let $D_p:=\{w_p=0\}$. Then as proved in \cite{Demailly2015}, the multiplier ideal sheaf $\mathcal{I}(\psi)$ is given by the direct image formula
\[\mathcal{I}(\psi)=\mu_\ast\mathcal{O}_{\widetilde{X}}(-\sum\limits_{p=1}^n\lfloor ca_p-b_p\rfloor_+D_p),\]
where $\lfloor ca_p-b_p\rfloor_+$ denotes the minimal nonnegative integer bigger than $ca_p-b_p-1$. Since $\psi$ has log canonical singularities, by the construction of $\mu$ and Lemma \ref{l:Hironaka}, one of the following cases is true on $W$:
\begin{enumerate}
\item[$(A)$]
$\widetilde{Y}$ is given on $W$ precisely by $D_{p_0}$ (if $W$ is small enough) for some $p_0$ satisfying $ca_{p_0}-b_{p_0}=1$, and $ca_p-b_p\leq1$ for $p\neq p_0$;
\item[$(B)$]
$\widetilde{Y}\cap W=\emptyset$, and $ca_p-b_p\leq1$.
\end{enumerate}
By definition, the measure $|f|^2d\lambda[\psi]$ can be defined as
\begin{equation}\label{e:prop2-limit}
g\mapsto\varlimsup\limits_{t\rightarrow-\infty}\int_{\{t<c\log|w^a|^2+\widetilde u(w)<t+1\}}\frac{|\widetilde{f}\circ\mu|^2(\widetilde{g}\circ\mu)\xi e^{-\widetilde u}}{|w^{ca-b}|^2}d\lambda(w),
\end{equation}
where $d\lambda(w):=$ the Lebesgue measure with respect to the coordinate vector $w$, $\widetilde{f}$ is a holomorphic extension of $f$ to $\Omega$, $g$ and $\widetilde{g}$ are defined as in Definition \ref{d:measure}, and $\xi$ is the smooth positive function $\frac{|J_\mu|^2}{|w^b|^2}$ (as stated in \cite{Demailly2015}, one would still have to take into account a partition of unity on the various coordinate charts covering the fibers of $\mu$, but we will avoid this technicality for the simplicity of notation).

In Case $(A)$, let us denote $w=(w',w_{p_0})\in \mathbb{C}^{n-1}\times\mathbb{C}$, $a=(a',a_{p_0})$, $b=(b',b_{p_0})$ and $d\lambda(w)=d\lambda(w')d\lambda(w_{p_0})$. Then $(\ref{e:prop2-limit})$ becomes
\[g\mapsto\varlimsup\limits_{t\rightarrow-\infty}\int_{\{t<c\log|w^a|^2+\widetilde u(w)<t+1\}}\frac{|\widetilde{f}\circ\mu|^2}{|(w')^{ca'-b'}|^2}\cdot
\frac{(\widetilde{g}\circ\mu)\xi e^{-\widetilde u}}{|w_{p_0}|^2}d\lambda(w).\]
Since the domain of integration can be written as
\[\big\{e^{t-\widetilde u(w)}|(w')^{a'}|^{-2c}<|w_{p_0}|^{2ca_{p_0}}<e^{t+1-\widetilde u(w)}|(w')^{a'}|^{-2c}\big\},\]
$(\ref{e:prop2-limit})$ becomes
\begin{equation}\label{e:prop2-limit1}
g\mapsto \frac{\pi}{ca_{p_0}}\int_{w'\in D_{p_0}}\frac{|f\circ\mu|^2}{|(w')^{ ca'-b'}|^2}\cdot(g\circ\mu)\xi e^{-\widetilde u}d\lambda(w').
\end{equation}

Set $\kappa=\{p:\,ca_p-b_p=1\}$.

If $p\in\kappa\setminus\{p_0\}$, then Theorem \ref{l:Hironaka} and
the construction of $\mu$ imply that an image of $D_p$ under a
finite sequence of blow-ups in the desingularization process must be
contained in a smooth center contained in $Y$ or $\mu_2^{-1}(|Y'|)$.
Hence the images of $D_p$ and $D_p\cap D_{p_0}$ coincide under the
composition of these blow-ups.

Since it is implied from $(\ref{ie:prop2-f-L2})$ and $(\ref{e:prop2-limit1})$ that $f\circ\mu\big|_{D_p\cap D_{p_0}}=0$, we obtain that
\begin{equation}\label{e:f-vanish}
f\circ\mu \big|_{D_p}=0
\end{equation}
holds for all $p\in \kappa\setminus\{p_0\}$ in Case $(A)$.

Similarly, we can get that $(\ref{e:f-vanish})$ holds for all $p\in\kappa$ in Case $(B)$. Then $(\ref{e:prop2-limit})$ is the zero measure in Case $(B)$.

Therefore, we represent the measure $|f|^2d\lambda[\psi]$ on $Y^0\cap U$ explicitly as in $(\ref{e:prop2-limit1})$.

\textbf{Step 2: we will obtain some uniform estimates for $f_t\circ\mu$.}

By Cauchy's inequality for holomorphic functions, it follows from $(\ref{ie:special-extension-1})$ that
\begin{equation}\label{ie:sup-norm-family-derivative}
\sup\limits_{U_1}|\partial^\gamma f_t|^2 \leq C_1\sup\limits_{V}|f_t|^2\leq
C_1C e^{-\beta_1t}
\end{equation}
for any $t\in(-\infty,0)$ and any multi-index $\gamma$ satisfying $|\gamma|\leq n$, where $U_1\subset\subset V$ is a neighborhood of $\overline{U}$, and $C_1$ is a positive number independent of $t$ and $\gamma$.

Let $W_t:=W\cap\mu^{-1}(U)\cap\{\psi\circ\mu<t+c_2\}$.

In Case $(A)$, by applying the mean value theorem to $f_t\circ\mu$ successively along the directions in $\kappa$, we get from $(\ref{ie:sup-norm-family-derivative})$ and $(\ref{e:f-vanish})$ that for any $w=(w',w_{p_0})\in W_t$,
\begin{eqnarray}
& &|f_t\circ\mu(w',w_{p_0})-f_t\circ\mu(w',0)|^2\label{ie:case1-ft}
\\&\leq&C_2\prod\limits_{p\in\kappa}|w_p|^2\sup\limits_{|\gamma|\leq|\kappa|}\sup\limits_{\mu^{-1}(U_1)}|\partial^\gamma f_t|^2\nonumber
\\&\leq&C_3e^{-\beta_1t}\prod\limits_{p\in\kappa}|w_p|^2\nonumber
\end{eqnarray}
and
\begin{equation}\label{ie:f-estimate}
|f_t\circ\mu(w',0)|^2=|f\circ\mu(w',0)|^2\leq C_4\prod\limits_{p\in\kappa\setminus\{p_0\}}|w_p|^2
\end{equation}
when $t$ is small enough, where $C_2$, $C_3$ and $C_4$ are positive numbers independent of $t$.

In Case $(B)$, if $\kappa\neq\emptyset$, take $p_1\in\kappa$ and denote $w=(w'',w_{p_1})$. Since $f_t\circ\mu(w'',0)=f\circ\mu(w'',0)=0$, by the similar method we have that
\begin{equation}\label{ie:case2-ft-1}
|f_t\circ\mu(w'',w_{p_1})|^2
\leq C_5e^{-\beta_1t}\prod\limits_{p\in\kappa}|w_p|^2
\end{equation}
for any $w=(w'',w_{p_1})\in W_t$ when $t$ is small enough, where $C_5$ is a positive number independent of $t$. If $\kappa=\emptyset$, $(\ref{ie:special-extension-1})$ implies that
\begin{equation}\label{ie:case2-ft-2}
|f_t\circ\mu(w)|^2\leq Ce^{-\beta_1t}
\end{equation}
for any $w\in W_t$.

\textbf{Step 3: the proof of $(\ref{ie:prop2-final-estimate})$.}

Let $j$ be a positive integer. Then $(\ref{ie:special-extension-2})$
implies that
\begin{eqnarray*}
& &\frac{1}{e^t}\int_{\{\phi\leq-j\}\cap
U\cap\{\psi<t+c_2\}}|f_t|^2e^{-\phi}d\lambda
\\&\leq&
\frac{1}{e^t}\int_{\{\phi\leq-j\}\cap
U\cap\{\psi<t+c_2\}}
|f_t|^2e^{-(1+\beta)\phi-\beta j}d\lambda
\\&\leq&C e^{-\beta j}
\end{eqnarray*}
for any $t\in(-\infty,0)$.

Therefore, for every $\epsilon\in(0,1)$, there exists a positive integer
$j_\epsilon$ such that
\begin{eqnarray}
& &\int_{\{\phi\leq-j_\epsilon\}\cap U\cap\{t-c_1<\psi<t+c_2\}}\frac{ e^tv|f_t|^2e^{-\phi}}
{(e^\psi+ e^t)^2}d\lambda\label{ie:singular-part-1}
\\&\leq&\frac{1}{(e^{-c_1}+1)^2e^t}\int_{\{\phi\leq-j_\epsilon\}\cap
U\cap\{\psi<t+c_2\}}v|f_t|^2e^{-\phi}d\lambda\nonumber
\\&\leq&
\frac{\beta_vC e^{-\beta
j_\epsilon}}{(e^{-c_1}+1)^2}\nonumber
\\&<&\frac{\epsilon}{2}\nonumber
\end{eqnarray}
for any $t\in(-\infty,0)$.

Set $\phi_\epsilon =\max\{\phi,-j_\epsilon\}$. We want to prove
\begin{equation}\label{ie:prop2-estimate-cutoff}
\varlimsup\limits_{ t\rightarrow-\infty }\int_{U\cap\{t-c_1<\psi<t+c_2\}}\frac{e^tv|f_t|^2e^{-\phi_\epsilon}} {(e^\psi+e^t)^2}d\lambda\leq\int_{U\cap Y^0 }v|f|^2e^{-\phi_\epsilon} d\lambda[\psi].
\end{equation}
Set
\[I_0=\varlimsup\limits_{ t\rightarrow-\infty }\int_{W\cap\mu^{-1}(U)\cap\{t-c_1<\psi\circ\mu<t+c_2\}}\frac{e^t(v\circ\mu)|f_t\circ\mu|^2e^{-\phi_\epsilon\circ\mu}|J_\mu|^2} {(e^{\psi\circ\mu}+e^t)^2}d\lambda.\]
Then by Step 1, it suffices to prove that
\begin{equation}\label{ie:prop2-final-estimate1}
I_0\leq\frac{\pi}{ca_{p_0}}\int_{W\cap\mu^{-1}(U)\cap D_{p_0}}\frac{(v\circ\mu)|f\circ\mu|^2\xi e^{-\widetilde u-\phi_\epsilon\circ\mu}}{|(w')^{ca'-b'}|^2}d\lambda(w')
\end{equation}
in Case $(A)$ and $I_0=0$ in Case $(B)$, where $\xi$ is the smooth positive function $\frac{|J_\mu|^2}{|w^b|^2}$ defined in Step 1.

In Case $(A)$, let
\[\Phi_t(w'):=\int_{W_{t,w'}}\frac{e^t(v\circ\mu)|f_t\circ\mu|^2e^{-\phi_\epsilon\circ\mu}|J_\mu|^2} {(e^{\psi\circ\mu}+e^t)^2}d\lambda(w_{p_0})\]
and
\[\Phi(w'):=\frac{\pi}{ca_{p_0}}\cdot\frac{v\circ\mu(w',0)|f\circ\mu(w',0)|^2\xi(w',0)e^{-\widetilde u(w',0)-\phi_\epsilon\circ\mu(w',0)}}{|(w')^{ca'-b'}|^2},\]
where $W_{t,w'}$ is the $1$-dimensional open set
\[\big\{e^{t-c_1-\widetilde{u}(w',w_{p_0})}|(w')^{a'}|^{-2c}<|w_{p_0}|^{2ca_{p_0}}<e^{t+c_2-\widetilde{u}(w',w_{p_0})}|(w')^{a'}|^{-2c}\big\}\cap W\cap\mu^{-1}(U)\]
for every fixed $t$ and $w'$ ($w'\in D_{p_0}\setminus\underset{p\neq p_0}\cup D_p$).
Then
\begin{equation}\label{ie:I_0}
I_0=\varlimsup\limits_{ t\rightarrow-\infty }\int_{W\cap\mu^{-1}(U)\cap D_{p_0}}\Phi_t(w')d\lambda(w').
\end{equation}

Since $-c_1<\psi\circ\mu-t<c_2$
holds on $W_{t,w'}$, we obtain from $(\ref{ie:case1-ft})$ and $(\ref{ie:f-estimate})$ that
\begin{eqnarray*}
\Phi_t(w')&\leq&C_6\int_{W_{t,w'}}\frac{(v\circ\mu)|f_t\circ\mu|^2e^{-\phi_\epsilon\circ\mu}|J_\mu|^2}{e^{\psi\circ\mu}}d\lambda(w_{p_0})
\\&\leq&C_7\int_{W_{t,w'}}\frac{|f_t\circ\mu|^2}{|w^{ca-b}|^2}d\lambda(w_{p_0})
\\&\leq&C_8\int_{W_{t,w'}}\frac{\prod\limits_{p\in\kappa}|w_p|^2}{|w^{(1+\beta_1)ca-b}|^2}d\lambda(w_{p_0})
+C_8\int_{W_{t,w'}}\frac{\prod\limits_{p\in\kappa\setminus\{p_0\}}|w_p|^2}{|w^{ca-b}|^2}d\lambda(w_{p_0}),
\end{eqnarray*}
where $C_7$ and $C_8$ are positive numbers independent of $t$.

Since it is easy to prove that the right-hand side of the above inequality is dominated by a function of $w'$ which is independent of $t$ and belongs to $L^1(W\cap\mu^{-1}(U)\cap D_{p_0})$ when
\[\beta_1<\min\limits_{\{p:\,a_p\neq0\}}\frac{1-(ca_p-b_p)+\lfloor ca_p-b_p\rfloor_+}{ca_p},\]
it follows from $(\ref{ie:I_0})$ and Fatou's lemma that
\begin{equation}\label{ie:I_0-Fatou}
I_0\leq\int_{W\cap\mu^{-1}(U)\cap D_{p_0}}\varlimsup\limits_{ t\rightarrow-\infty }\Phi_t(w')d\lambda(w').
\end{equation}

Since $(\ref{ie:case1-ft})$ implies that
\[\lim\limits_{ t\rightarrow-\infty }\sup\limits_{w_{p_0}\in W_{t,w'}}|f_t\circ\mu(w',w_{p_0})-f\circ\mu(w',0)|=0\]
for every fixed $w'\in \big(W\cap\mu^{-1}(U)\cap D_{p_0}\big)\setminus\underset{p\neq p_0}\cup(D_{p_0}\cap D_p)$ when $\beta_1<1/ca_{p_0}$, it follows from Lemma \ref{l:one dim lemma} that
\[\varlimsup\limits_{ t\rightarrow-\infty }
\Phi_t(w')\leq\Phi(w'),\quad\forall
w'\in \big(W\cap\mu^{-1}(U)\cap D_{p_0}\big)\setminus\underset{p\neq p_0}\cup(D_{p_0}\cap D_p).\]
Hence $(\ref{ie:prop2-final-estimate1})$ follows from $(\ref{ie:I_0-Fatou})$. Similarly, we can obtain from $(\ref{ie:case2-ft-1})$ and $(\ref{ie:case2-ft-2})$ that $I_0=0$ in Case $(B)$ when
\[\beta_1<\min\limits_{\{p:\,a_p\neq0\}}\frac{1-(ca_p-b_p)+\lfloor ca_p-b_p\rfloor_+}{ca_p}.\]
Thus we get $(\ref{ie:prop2-estimate-cutoff})$.

It is easy to see that $(\ref{ie:prop2-final-estimate})$ follows from $(\ref{ie:singular-part-1})$ and $(\ref{ie:prop2-estimate-cutoff})$. Thus we finish the proof of Proposition \ref{p:convergence-integral-psh}.

\end{proof}


\section{Proof of Theorem \ref{t:Zhou-Zhu1}}\label{section-proof-of-theorem}

Without loss of generality, we can suppose that $f$ is not $0$ identically.

Let $h_0$ be any fixed smooth metric of $L$ on $X$. Then
$h=h_0e^{-\phi}$ for some global function $\phi$ on $X$, which is quasi-plurisubharmonic
by the assumption in the theorem.

Since $X$ is weakly pseudoconvex, there exists a smooth
plurisubharmonic exhaustion function $P$ on $X$. Let $X_k:=\{P<k\}$
($k=1,2,\cdots$, we choose $P$ such that $X_1\neq\emptyset$).

Our proof consists of several steps. We will discuss for fixed $k$
until the end of Step 5.

We will give the proof for the line bundle $L$ in the first five steps, and we will give the proof for the vector bundle $E$ in Step 6.

\textbf{Step 1: construction of a family of special smooth
extensions $\tilde{f}_t$ of $f$ to a neighborhood of
$\overline{X_k}\cap Y$ in $X$.}

In order to deal with singular metrics of holomorphic line bundles
on weakly pseudoconvex K\"{a}hler manifolds, we construct in this
step a family of smooth extensions $\tilde{f}_t$ of $f$
satisfying some special estimates by using the results in Section
\ref{section-singular-metric}.

Let $\epsilon\in(0,\frac{1}{2})$.

For the sake of clearness, we divide this step into four parts.

\textbf{Part I: construction of local coordinate patches
$\{\Omega_i\}_{i=1}^N$, $\{U_i\}_{i=1}^N$ and a partition of unity
$\{\xi_i\}_{i=1}^{N+1}$.}

For any point $x\in Y$, we can find a local coordinate ball
$\Omega_x'$ in $X$ centered at $x$ such that there exists a local holomorphic frame of
$L$ on $\Omega_x'$ and such that $\phi$
can be written as a sum of a smooth function and a plurisubharmonic
function on $\Omega_x'$. Moreover, we assume that $\psi$ can be written on $\Omega_x'$ as
\begin{equation}\label{e:Psi-local}
\psi=c_x\log\sum\limits_{1\leq j\leq j_0}|g_{x,j}|^2+u_x,
\end{equation}
where $c_x$ is a positive number, $g_{x,j}\in\mathcal{O}_X(\Omega_x')$ and $u_x\in C^\infty(\Omega_x')$.

Let $U_x\subset\subset V_x\subset\subset\Omega_x\subset\subset\Omega_x'$ be three small coordinate balls.

Since $\overline {X_k}\cap Y$ is compact, there exist points
$x_1,x_2,\cdots,x_N\in \overline {X_k}\cap Y$ such that $\overline
{X_k}\cap Y\subset \overset{N}{\underset{i=1}\cup} U_{x_i}$.

For simplicity, we will denote $\Omega_{x_i}'$, $\Omega_{x_i}$, $U_{x_i}$, $V_{x_i}$ and $u_{x_i}$ by $\Omega_i'$, $\Omega_i$,
$U_i$, $V_i$ and $u_i$ respectively. We will write the local expression $(\ref{e:Psi-local})$ on $\Omega_i'$ by
\[\psi=\Upsilon_i+u_i.\]

Choose an open set $U_{N+1}$ in $X$ such that $\overline {X_k}\cap
Y\subset X\setminus\overline {
U_{N+1}}\subset\subset\overset{N}{\underset{i=1}\cup} U_i$. Set
$U=X\setminus\overline {U_{N+1}}$.

Let $\{\xi_i\}_{i=1}^{N+1}$ be a partition of unity subordinate to
the cover $\{U_i\}_{i=1}^{N+1}$ of $X$. Then
$\mathrm{supp}\,\xi_i\subset\subset U_i$ for $i=1,\cdots,N$ and
$\sum\limits_{i=1}^N\xi_i=1$ on $U$.

\textbf{Part II: construction of local holomorphic extensions
$\widehat{f}_{i,t}$ $(1\leq i\leq N)$ of $f$ to
$\Omega_i\cap\{\psi<t+c_2\}$, where $c_2$ will be defined in
this part.}

By Remark \ref{r:thm-Stein}, $f$ has local $L^2$ extensions to local coordinate balls around every point in $Y$. Hence $f$ is indeed a holomorphic section well defined on $Y$ (not only on $Y^0$). By Step 1 (see $(\ref{e:prop2-limit1})$) in the proof of Proposition \ref{p:convergence-integral-psh}, $(\ref{ie:thm-f-finite})$ is equivalent to
\[\int_{D_{p_0}}\frac{|f\circ\mu|^2_{\omega,h_0}\xi e^{-\widetilde u-\phi\circ\mu}}{|(w')^{ ca'-b'}|^2}d\lambda(w')<+\infty.\]
Hence by Theorem \ref{l:strong-openness}, there exists a positive
number $\beta\in(0,1)$ such that
\begin{equation}\label{ie:f-finite}
\int_{\Omega_i\cap Y^0}|f|^2_{\omega,h_0}e^{-(1+\beta)\phi}dV_{X,\omega}[\psi]<+\infty\quad(1\leq i\leq N).
\end{equation}

Let $\widetilde{\alpha}_0<\alpha_0$ be a fixed number such that $R$ is decreasing on $(-\infty,\widetilde{\alpha}_0]$. Then set $R_0(t)=R(\widetilde{\alpha}_0)e^{-\beta_2(t-\widetilde{\alpha}_0)}$,
$t\in(-\infty,\widetilde{\alpha}_0]$, where $\beta_2$ is a positive number which will be determined later in Step 4.
Let
\[R_1(t):=\min\{R_0(t+\widetilde{\alpha}_0),R(t+\widetilde{\alpha}_0)\},\quad t\in(-\infty,0].\]
Then $R_1$ is decreasing and thereby satisfies all the requirements for the functions in $\mathfrak{R}_{0,\alpha_1}$ except that $R_1$ is only continuous.

Let $c_1=c_2:=\log\frac{2-\epsilon}{\epsilon}$, $m_i:=\inf\limits_{\Omega_i}u_i$ and $M_i:=\sup\limits_{\Omega_i}u_i$.

For each fixed $t\in(-\infty,0)$, by Remark \ref{r:thm-Stein}, we
apply Theorem \ref{t:Zhou-Zhu1} to the Stein manifold $\Omega_i\cap\{\Upsilon_i<t+c_2-m_i\}$, to the negative plurisubharmonic function $\Upsilon_i-t-c_2+m_i$, to the
holomorphic section $f$ on $\Omega_i\cap Y^0$ with the $L^2$ condition
$(\ref{ie:f-finite})$ and to the function $R_1$ ($R_1$ is only needed to be continuous by the remark after Theorem 2.1 in \cite{Guan-Zhou2013b}), and then we obtain $L^2$ extensions of $f$
from $\Omega_i\cap Y^0$ to
\[\Omega_i\cap\{\Upsilon_i<t+c_2-m_i\},\]
where we equip the line bundle $L$ with the singular metric
$h_0e^{-(1+\beta)\phi}$. More precisely, there exists a uniform
positive number $C_1$ (independent of $t$) and
holomorphic extensions $\widehat{f}_{i,t}$ ($1\leq i\leq
N$) of $f$ from $\Omega_i\cap Y^0$ to $\Omega_i\cap\{\Upsilon_i<t+c_2-m_i\}$
such that
\begin{eqnarray}
&
&\int_{\Omega_i\cap\{\Upsilon_i<t+c_2-m_i\}}\frac{|\widehat{f}_{i,t}|^2_{\omega,h_0}e^{-(1+\beta)\phi}}
{e^{\Upsilon_i-t-c_2+m_i}R_1(\Upsilon_i-t-c_2+m_i)}
dV_{X,\omega}\label{ie:widehat-f-estimate}\\
&\leq& C_1\int_{\Omega_i\cap
Y^0}|f|^2_{\omega,h_0}e^{-(1+\beta)\phi}dV_{X,\omega}[\Upsilon_i-t-c_2+m_i]\nonumber\\
&\leq&C_2e^t\int_{\Omega_i\cap
Y^0}|f|^2_{\omega,h_0}e^{-(1+\beta)\phi}dV_{X,\omega}[\psi],\nonumber
\end{eqnarray}
where $C_2$ is a positive number independent of $t$. Furthermore, we get that $f$ is in fact holomorphic on $\Omega_i\cap Y$ and $\widehat{f}_{i,t}=f$ on $\Omega_i\cap Y$.

\textbf{Part III: construction of local holomorphic extensions
$\tilde{f}_{i,t}$ $(1\leq i\leq N)$ of $f$ to $\Omega_i$.}

For each fixed $t$,
applying Proposition \ref{p:special-L2-extension} to the local
extensions $\widehat{f}_{i,t}$ $(1\leq i\leq N)$ with the
weight $(1+\beta)\phi$ and to the case
$\Upsilon=\Upsilon_i-t-c_2+m_i$, $\Omega=\Omega_i$ and
some small positive number $\beta_1$ which will be determined later in Step 4, we obtain from
$(\ref{ie:widehat-f-estimate})$ holomorphic sections
$\tilde{f}_{i,t}$ $(1\leq i\leq N)$ on $\Omega_i$ satisfying
$\tilde{f}_{i,t}=\widehat{f}_{i,t}=f$ on
$\Omega_i\cap Y^0$,
\begin{equation}\label{ie:special-extension--1}
\int_{\Omega_i\cap\{\Upsilon_i<t+c_2-m_i\}}\frac{|\tilde{f}_{i,t}|^2_{\omega,h_0}e^{-(1+\beta)\phi}}
{e^{\Upsilon_i-t-c_2+m_i}R_1(\Upsilon_i-t-c_2+m_i)}dV_{X,\omega}\leq
C_3e^t,
\end{equation}
and
\begin{equation}\label{ie:special-extension--2}
\int_{\Omega_i}\frac{|\tilde{f}_{i,t}|^2_{\omega,h_0}e^{-(1+\beta)\phi}}
{(1+e^{\Upsilon_i-t-c_2+m_i})^{1+\beta_1}}dV_{X,\omega}\leq
C_3e^t
\end{equation}
for some positive number $C_3$ independent of $t$.

Since $\sup\limits_{t\leq0}\big(e^tR_1(t)\big)<+\infty$, it follows
from $(\ref{ie:special-extension--1})$ that
\begin{equation}\label{ie:special-extension-a}
\int_{\Omega_i\cap\{\psi<t+c_2\}}
|\tilde{f}_{i,t}|^2_{\omega,h_0}e^{-(1+\beta)\phi}
dV_{X,\omega}\leq C_4e^t
\end{equation}
for any $t$, where
$C_4$ is a positive number independent of $t$.

Since $\Upsilon_i$ is bounded above on $\Omega_i$, it follows from
$(\ref{ie:special-extension--2})$ that
\begin{equation}\label{ie:special-extension-b}
\int_{\Omega_i}
|\tilde{f}_{i,t}|^2_{\omega,h_0}e^{-(1+\beta)\phi}
dV_{X,\omega}\leq C_5e^{-\beta_1t}
\end{equation}
for any $t$, where
$C_5$ is a positive number independent of $t$.

Since $|\tilde{f}_{i,t}|^2$ is subharmonic on $\Omega_i$,
by mean value inequality, we get
from $(\ref{ie:special-extension-b})$ that
\begin{equation}\label{ie:special-extension-bb}
\sup\limits_{V_i}
|\tilde{f}_{i,t}|^2_{\omega,h_0}\leq C_6e^{-\beta_1t}
\end{equation}
for any $t$, where
$C_6$ is a positive number independent of $t$.

Since $(\ref{ie:special-extension-a})$ and
$(\ref{ie:special-extension-bb})$ imply that the assumptions in
Proposition \ref{p:convergence-integral-psh} hold for
$\tilde{f}_{i,t}$, we apply Proposition \ref{p:convergence-integral-psh} to
$\tilde{f}_{i,t}$ $(1\leq i\leq N)$ and get
\begin{eqnarray}
& &\varlimsup\limits_{ t\rightarrow-\infty }
\int_{U_i\cap\{t-c_1<\psi<t+c_2\}}
\frac{ e^t\xi_i|\tilde{f}_{i,t}|^2_{\omega,h_0}
e^{-\phi}}{(e^\psi+ e^t)^2}dV_{X,\omega}\label{ie:family-integral-limit}\\
&\leq&\int_{U_i\cap Y^0}\xi_i|f|^2_{\omega,h_0} e^{-\phi}dV_{X,\omega}[\psi],\nonumber
\end{eqnarray}
which will be used in Step 4.

\textbf{Part IV: construction of a family of smooth extensions
$\tilde{f}_t$ of $f$ to a neighborhood of
$\overline{X_k}\cap Y$ in $X$.}

Define
$\tilde{f}_t=\sum\limits_{i=1}^N\xi_i\tilde{f}_{i,t}$
for all $t$.

Since
\[\tilde{f}_t|_{U_{j}}=\sum\limits_{i=1}^N\xi_{i}
\tilde{f}_{j,t}+\sum\limits_{i=1}^N\xi_{i}(\tilde{f}_{i,t}-
\tilde{f}_{j,t})
=\tilde{f}_{j,t}+\sum\limits_{i=1}^N\xi_{i}(\tilde{f}_{i,t}-
\tilde{f}_{j,t})\]
 for any $j=1,\cdots,N$, we have
\begin{equation}\label{e:dbar of smooth extension}
|\mathrm{D}''\tilde{f}_t|_{\omega,h_0}\big|_{U_j} =|\sum\limits_{i=1}^N\bar\partial\xi_{i}\wedge(\tilde{f}_{i,t}- \tilde{f}_{j,t})
|_{\omega,h_0},\quad\forall t.
\end{equation}

Let $\mu$ and $W$ be as in the beginning of the proof of Proposition \ref{p:convergence-integral-psh} (here $W$ is centered at a point $\widetilde{x}\in\overline{\mu^{-1}(U_i\cap U_j)}\cap\{\psi=-\infty\}$). For similar reasons as in $(\ref{ie:case1-ft})$, $(\ref{ie:case2-ft-1})$
and $(\ref{ie:case2-ft-2})$, we get from $(\ref{ie:special-extension-bb})$ that
\begin{equation}
|\tilde{f}_{i,t}\circ\mu-\tilde{f}_{j,t}\circ\mu|^2_{\omega,h_0}
\big|_{W_{i,j,t}}\leq C_7e^{-\beta_1t}\prod\limits_{p\in\kappa}|w_p|^2\label{ie:estimate
of smooth extension1}
\end{equation}
when $\kappa\neq\emptyset$ and $t$ is small enough, and that
\begin{equation}\label{ie:estimate of smooth extension2}
|\tilde{f}_{i,t}\circ\mu-\tilde{f}_{j,t}\circ\mu|^2_{\omega,h_0}
\big|_{W_{i,j,t}}\leq C_7e^{-\beta_1t}
\end{equation}
when $\kappa=\emptyset$ and $t$ is small enough, where
\[W_{i,j,t}:=W\cap\mu^{-1}(U_i\cap U_j)\cap\{\psi\circ\mu<t+c_2\}\]
and $C_7$ is a positive number independent of $t$.

\textbf{Step 2: singularity attenuation process for the current
$\sqrt{-1}\partial\bar\partial\phi$.}

Since the singularities of $\sqrt{-1}\partial\bar\partial\psi$
obstruct the application of Lemma \ref{l:Demailly1994}, we will work on $\widetilde{X}$ first and then go back to $X$. Some ideas in this step come from \cite{Yi2}.

Let $\mu:\widetilde{X}\rightarrow X$
be as in the beginning of the proof of Proposition \ref{p:convergence-integral-psh}. Let
$\widetilde{X}_{k+1}:=\mu^{-1}(X_{k+1})$,
$\widetilde{X}_k:=\mu^{-1}(X_k)$ and $\widetilde{\Sigma_0}:=\mu^{-1}(\Sigma_0)$, where $\Sigma_0:=\{\psi=-\infty\}$. Then
\[\gamma_1:=\sqrt{-1}\partial\bar\partial(\psi\circ\mu)-\sum\limits_{j} q_j [D_j]\]
is a smooth real $(1,1)$-form for some positive numbers $ q_j $, where $(D_j)$ are the irreducible components of $\widetilde{\Sigma_0}$.
It is not hard to prove the following lemma and we won't give its proof.

\begin{lem}\label{l:Kahler metric under resolution}
There exists a positive number $\widetilde{n}_k$ such that
\[\widetilde{\omega}_{k+1}:=\widetilde{n}_k\mu^*\omega+
\sqrt{-1}\partial\bar\partial(\psi\circ\mu)-\sum\limits_j q_j [D_j]\]
is a K\"{a}hler metric on $\widetilde{X}_{k+1}$.
\end{lem}

Since $\mu:\widetilde{X}\setminus\widetilde{\Sigma_0}\rightarrow X\setminus
\Sigma_0$ is biholomorphic and
$\sum\limits_j q_j [D_j]\big|_{\widetilde{X}
\setminus\widetilde{\Sigma_0}}=0$,
the curvature assumptions $(i)$ and $(ii)$ in Theorem
\ref{t:Zhou-Zhu1} implies that
\[\sqrt{-1}\partial\bar\partial(\phi\circ\mu)
\big|_{\widetilde{X}\setminus\widetilde{\Sigma_0}}+\gamma_2
\big|_{\widetilde{X}\setminus\widetilde{\Sigma_0}}\geq 0\]
and
\[\sqrt{-1}\partial\bar\partial(\phi\circ\mu)
\big|_{\widetilde{X}\setminus\widetilde{\Sigma_0}}+\gamma_3
\big|_{\widetilde{X}\setminus\widetilde{\Sigma_0}}\geq0\] hold on
$\widetilde{X}\setminus\widetilde{\Sigma_0}$, where
\[\gamma_2:=\sqrt{-1}\mu^*\Theta_{L,h_0}+\gamma_1,
\quad\gamma_3:=\sqrt{-1}\mu^*\Theta_{L,h_0}+\bigg(1+\frac{1}{\widetilde{\chi}(\alpha\circ\mu)}\bigg)\gamma_1.\]
Since
$\gamma_2$ and $\gamma_3$ are continuous on
$\widetilde{X}$, and $\phi\circ\mu$ is quasi-plurisubharmonic on
$\widetilde{X}$, we get that
\begin{equation}\label{ie:curvature bound1}\sqrt{-1}\partial\bar\partial(\phi\circ\mu)
+\gamma_2\geq 0
\end{equation}
and
\begin{equation}\label{ie:curvature bound2}
\sqrt{-1}\partial\bar\partial(\phi\circ\mu) +\gamma_3\geq0
\end{equation}
hold on $\widetilde{X}$. Since there must exist a continuous
nonnegative $(1,1)$-form $\varpi_{k+1}$ on the K\"{a}hler manifold
$(\widetilde{X}_{k+1},\widetilde{\omega}_{k+1})$ such that
\[(\sqrt{-1}\Theta_{T_{\widetilde{X}_{k+1}}}+\varpi_{k+1}\otimes \mathrm{Id}_{T_{\widetilde{X}_{k+1}}})
(\kappa_1\otimes \kappa_2,\kappa_1\otimes \kappa_2)\geq0\quad
(\forall\kappa_1,\kappa_2\in T_{\widetilde{X}_{k+1}})\] holds on
$\widetilde{X}_{k+1}$, by Theorem \ref{l:Demailly1994}, we obtain
from $(\ref{ie:curvature bound1})$ and $(\ref{ie:curvature bound2})$
a family of functions
$\{\widetilde{\phi}_{\varsigma,\rho}\}_{\varsigma>0,
\rho\in(0,\rho_1)}$ on a neighborhood of the closure of
$\widetilde{X}_k$ such that
\begin{enumerate}
\item[$(i)$]
$\widetilde{\phi}_{\varsigma,\rho}$ is quasi-plurisubharmonic on a neighborhood of
the closure of $\widetilde{X}_k$, smooth on $\widetilde{X}_{k}
\setminus E_\varsigma(\phi\circ\mu)$, increasing with respect to
$\varsigma$ and $\rho$ on $\widetilde{X}_k$, and converges to
$\phi\circ\mu$ on $\widetilde{X}_k$ as $\rho\rightarrow0$,
\item[$(ii)$]
$\frac{\sqrt{-1}}{\pi}\partial\bar\partial
\widetilde{\phi}_{\varsigma,\rho}\geq-\frac{\gamma_2}{\pi}
-\varsigma \varpi_{k+1}-\delta_\rho\widetilde{\omega}_{k+1}$ on
$\widetilde{X}_k$,
\item[$(iii)$]
$\frac{\sqrt{-1}}{\pi}\partial\bar\partial
\widetilde{\phi}_{\varsigma,\rho}\geq-\frac{\gamma_3}{\pi}
-\varsigma \varpi_{k+1}-\delta_\rho\widetilde{\omega}_{k+1}$ on
$\widetilde{X}_k$,
\end{enumerate}
where $E_\varsigma(\phi\circ\mu):=\{x\in \widetilde{X}:\,
\nu(\phi\circ\mu,x)\geq \varsigma\}$ $(\varsigma>0)$ is the
$\varsigma$-upperlevel set of Lelong numbers of $\phi\circ\mu$, and
$\{\delta_\rho\}$ is an increasing family of positive numbers such
that $\lim\limits_{\rho\rightarrow 0}\delta_\rho=0$.

Since $\widetilde{\omega}_{k+1}$ is a K\"{a}hler metric on
$\widetilde{X}_{k+1}$ by Lemma \ref{l:Kahler metric under resolution} and
$\widetilde{X}_k$ is relatively compact in $\widetilde{X}_{k+1}$,
there exists a positive number $n_k>1$ such that
$n_k\widetilde{\omega}_{k+1}\geq\varpi_{k+1}$ holds on
$\widetilde{X}_k$. Take $\varsigma=\delta_\rho$ and denote
$\widetilde{\phi}_{\delta_\rho,\rho}$ simply by
$\widetilde{\phi}_\rho$. Then $\widetilde{\phi}_\rho$ is
quasi-plurisubharmonic on a neighborhood of the closure of
$\widetilde{X}_k$, smooth on $\widetilde{X}_{k}\setminus
E_{\delta_\rho}(\phi\circ\mu)$, increasing with respect to $\rho$ on
$\widetilde{X}_k$, and converges to $\phi\circ\mu$ on
$\widetilde{X}_k$ as $\rho\rightarrow0$. Furthermore,
\[\sqrt{-1}\partial\bar\partial\widetilde{\phi}_\rho
+\gamma_2+2\pi n_k\delta_\rho\widetilde{\omega}_{k+1} \geq0\]
and
\[\sqrt{-1}\partial\bar\partial\widetilde{\phi}_\rho
+\gamma_3+2\pi n_k\delta_\rho\widetilde{\omega}_{k+1} \geq0\] hold
on $\widetilde{X}_k$. Since $\mu:\widetilde{X}_k
\setminus\widetilde{\Sigma_0}\rightarrow X_k\setminus \Sigma_0$ is biholomorphic,
we get that
\[\sqrt{-1}\partial\bar\partial(\widetilde{\phi}_\rho
\circ\mu^{-1})+(\mu^{-1})^*\gamma_2+2\pi n_k\delta_\rho
(\mu^{-1})^*\widetilde{\omega}_{k+1}\geq 0\]
and
\[\sqrt{-1}\partial\bar\partial(\widetilde{\phi}_\rho
\circ\mu^{-1})+(\mu^{-1})^*\gamma_3+2\pi n_k\delta_\rho
(\mu^{-1})^*\widetilde{\omega}_{k+1}\geq 0\] hold on $X_k\setminus
\Sigma_0$. Then, replacing $\gamma_2$, $\gamma_3$ and
$\widetilde{\omega}_{k+1}$ with their definitions, we obtain that
\begin{equation}\label{ie:regularization of
curvature1}
\sqrt{-1}\partial\bar\partial
(\widetilde{\phi}_\rho\circ\mu^{-1}) +\sqrt{-1}\Theta_{L,h_0}+(1+2\pi
n_k\delta_\rho)\sqrt{-1}
\partial\bar\partial\psi\geq-2\pi n_k\widetilde{n}_k\delta_\rho\omega
\end{equation}
and
\begin{equation}\label{ie:regularization of
curvature2}
\sqrt{-1}\partial\bar\partial
(\widetilde{\phi}_\rho\circ\mu^{-1}) +\sqrt{-1}\Theta_{L,h_0}+\bigg(1+2\pi
n_k\delta_\rho+\frac{1}{\widetilde{\chi}(\alpha)}\bigg)\sqrt{-1}
\partial\bar\partial\psi\geq-2\pi
n_k\widetilde{n}_k \delta_\rho\omega
\end{equation}
hold on $X_k\setminus \Sigma_0$.

Since $E_{\delta_\rho}(\phi\circ\mu)$ is an analytic set in
$\widetilde{X}$, Remmert's proper mapping theorem implies that
\[\Sigma_\rho:=\mu\big(E_{\delta_\rho} (\phi\circ\mu)\big)\] is an
analytic set in $X$. By Lemma \ref{l:Demailly1982-complete-metric},
$X_k\setminus(\Sigma_0\cup \Sigma_\rho)$ is a complete K\"{a}hler manifold.

It follows from the properties of $\widetilde{\phi}_\rho$ that
$\widetilde{\phi}_\rho\circ\mu^{-1}$ is smooth on
$X_{k}\setminus(\Sigma_0\cup \Sigma_\rho)$, increasing with respect to
$\rho$ on $X_k\setminus \Sigma_0$, uniformly bounded above on $X_k\setminus
\Sigma_0$ with respect to $\rho$, and converges to $\phi$ on $X_k\setminus
\Sigma_0$ as $\rho\rightarrow0$.

In Step 3, we will use $\widetilde{\phi}_\rho\circ\mu^{-1}$ to
construct a smooth metric of $L$ on $X_k\setminus(\Sigma_0\cup
\Sigma_\rho)$.

\textbf{Step 3: construction of additional weights and twist factors.}

Let $\zeta$, $\chi$ and $\eta$ be the solution to the following system of ODEs
defined on $(-\infty,\alpha_0)$:
\begin{numcases}{}
\chi(t)\zeta'(t)-\chi'(t)=1,\label{5 ode1}\\
\big(\chi(t)+\eta(t)\big)e^{\zeta(t)}=\bigg(\frac{\alpha_1}{R(\alpha_0)}+C_R\bigg)R(t),\label{5 ode2}\\
\frac{\big(\chi'(t)\big)^2}{\chi(t)
\zeta''(t)-\chi''(t)}=\eta(t),\label{5 ode3}
\end{numcases}
where we assume that $\zeta$, $\chi$ and $\eta$ are all
smooth on $(-\infty,\alpha_0)$, and that $\inf\limits_{t<\alpha_0}\zeta(t)=0$, $\inf\limits_{t<\alpha_0}\chi(t)=\alpha_1$, $\eta>0$,
$\zeta'>0$ and $\chi'<0$ on
$(-\infty,\alpha_0)$. If $\alpha_0=+\infty$, we replace the assumption $\inf\limits_{t<\alpha_0}\chi(t)=\alpha_1$ by $\chi>0$. By the similar calculation as in \cite{Guan-Zhou2013b} or \cite{Zhou-Zhu2015}, we can solve the system of ODEs and the solution is
\begin{numcases}{}
\chi(t)=\widetilde{\chi}(t),\nonumber\\
\zeta(t)=\log\bigg(\frac{\alpha_1}{R(\alpha_0)}+C_R\bigg)-\log\bigg(\frac{\alpha_1}{R(\alpha_0)}+\int_t^{\alpha_0}\frac{dt_1}{R(t_1)}\bigg),\nonumber\\
\eta(t)=R(t)\bigg(\frac{\alpha_1}{R(\alpha_0)}+\int_t^{\alpha_0}\frac{dt_1}{R(t_1)}\bigg)-\widetilde{\chi}(t),\nonumber
\end{numcases}
where $\widetilde{\chi}(t)$ is defined by $(\ref{e:definition of chi})$.

Let $\epsilon\in(0,\frac{1}{2})$ be as in Step 1 and put $\sigma_{t}=
\log(e^\psi+e^t)-\epsilon$. Then there exists a negative number
$t_\epsilon$ such that
$\sigma_{t}\leq \alpha-\frac{\epsilon}{2}$ on $\overline
{X_k}$ for any $t\in(-\infty,t_\epsilon)$.

Let $h_{\rho,t}$ be the new metric on the line bundle $L$ over
$X_k\setminus(\Sigma_0\cup \Sigma_\rho)$ defined by
\[h_{\rho,t}:=
h_0e^{-\widetilde{\phi}_\rho\circ\mu^{-1} -(1+2\pi
n_k\delta_\rho)\psi-\zeta(\sigma_{t})}.\]

Let $\tau_t:=\chi(\sigma_{t})$ and
$A_t:=\eta(\sigma_{t})$. Set
$\mathrm{B}_{\rho,t} =[\Theta_{\rho,t},\,\Lambda
]$ on $X_k\setminus(\Sigma_0\cup \Sigma_\rho)$, where
\[\Theta_{\rho,t}:=\tau_t\sqrt{-1}
\Theta_{L,h_{\rho,t}}
-\sqrt{-1}\partial\bar\partial\tau_t
-\sqrt{-1}\frac{\partial\tau_t\wedge
\bar\partial\tau_t}{A_t}.\] Set $\nu_t
=\partial\sigma_t$. We want to prove
\begin{equation}\label{ie:twisted cuvature}
\Theta_{\rho,t}\big|_{X_k\setminus(\Sigma_0\cup \Sigma_\rho)}\geq
\frac{e^t}{ e^\psi }\sqrt{-1} \nu_t\wedge
\bar\nu_t-2\pi
n_k\widetilde{n}_k\chi(\sigma_{t})\delta_\rho\omega.
\end{equation}

It follows from $(\ref{5 ode1})$ and $(\ref{5 ode3})$ that
\begin{eqnarray*}
& &\Theta_{\rho,t}\big|_{X_k\setminus(\Sigma_0\cup \Sigma_\rho)}\\
&=&\chi(\sigma_{t})\big(\sqrt{-1}\Theta_{L,h_0}
+\sqrt{-1}\partial\bar\partial
(\widetilde{\phi}_\rho\circ\mu^{-1})+(1+2\pi
n_k\delta_\rho)\sqrt{-1}\partial\bar\partial\psi\big)\\
& &+\big(\chi(\sigma_{t})
\zeta'(\sigma_{t})
-\chi'(\sigma_{t})\big)
\sqrt{-1}\partial\bar\partial\sigma_{t}\\
& &+\bigg(\chi(\sigma_{t})
\zeta''(\sigma_{t})-\chi''(\sigma_{t})
-\frac{\big(\chi'(\sigma_{t})\big)^2}
{\eta(\sigma_{t})}\bigg)\sqrt{-1}
\partial\sigma_{t}
\wedge\bar\partial\sigma_{t}\\
&=&\chi(\sigma_{t})\big(\sqrt{-1}\Theta_{L,h_0}
+\sqrt{-1}\partial\bar\partial
(\widetilde{\phi}_\rho\circ\mu^{-1})+(1+2\pi
n_k\delta_\rho)\sqrt{-1}\partial\bar\partial\psi\big)+\sqrt{-1}\partial\bar\partial
\sigma_{t}\\
&=&\chi(\sigma_{t})\big(\sqrt{-1}\Theta_{L,h_0}
+\sqrt{-1}\partial\bar\partial
(\widetilde{\phi}_\rho\circ\mu^{-1})+(1+2\pi n_k\delta_\rho)\sqrt{-1}\partial\bar\partial\psi\big)\\
& &+\frac{e^t}{ e^\psi }\sqrt{-1} \nu_t\wedge
\bar\nu_t+\frac{e^\psi}{e^\psi+e^t}\sqrt{-1}\partial\bar\partial\psi.
\end{eqnarray*}
Since $\chi$ is decreasing and $\chi=\widetilde{\chi}$, it follows from $(\ref{ie:regularization
of curvature1})$ and $(\ref{ie:regularization of curvature2})$ that
\begin{eqnarray*}
& &\chi(\sigma_{t})\big(\sqrt{-1}\Theta_{L,h_0}
+\sqrt{-1}\partial\bar\partial
(\widetilde{\phi}_\rho\circ\mu^{-1})+(1+2\pi n_k\delta_\rho)\sqrt{-1}\partial\bar\partial\psi\big)\\
& &+\frac{e^\psi}{e^\psi+e^t}\sqrt{-1}\partial\bar\partial\psi\\
&=&\chi(\sigma_{t})\big(\sqrt{-1}\Theta_{L,h_0}
+\sqrt{-1}\partial\bar\partial
(\widetilde{\phi}_\rho\circ\mu^{-1})+(1+2\pi n_k\delta_\rho)\sqrt{-1}\partial\bar\partial\psi+2\pi n_k\widetilde{n}_k\delta_\rho\omega\big)\\
& & -2\pi
n_k\widetilde{n}_k\chi(\sigma_{t})\delta_\rho\omega
+\frac{\chi(\alpha) e^\psi}{e^\psi+e^t}\cdot\frac{\sqrt{-1}\partial\bar\partial\psi}{\chi(\alpha)}\\
&\geq&\frac{\chi(\alpha) e^\psi }{ e^\psi + e^t}
\bigg(\sqrt{-1}\Theta_{L,h_0}
+\sqrt{-1}\partial\bar\partial
(\widetilde{\phi}_\rho\circ\mu^{-1})+(1+2\pi n_k\delta_\rho)\sqrt{-1}\partial\bar\partial\psi\\
& &+2\pi
n_k\widetilde{n}_k\delta_\rho\omega+\frac{\sqrt{-1}\partial\bar\partial\psi}{\chi(\alpha)}\bigg)
-2\pi n_k\widetilde{n}_k\chi(\sigma_{t})\delta_\rho\omega\\
&\geq&-2\pi
n_k\widetilde{n}_k\chi(\sigma_{t})\delta_\rho\omega
\end{eqnarray*}
on $X_k\setminus(\Sigma_0\cup\Sigma_\rho)$. Hence we get $(\ref{ie:twisted
cuvature})$ as desired.

Let $\beta$ be as in Step 1. Let
$\beta_0$ and $\beta_3$ be two positive numbers which will be determined later in Step 4. We choose an increasing family of
positive numbers
$\{\rho_t\}_{t\in(-\infty,t_\epsilon)}$ such that
$\lim\limits_{ t\rightarrow-\infty }\rho_t=0$ and for any $t$,
\begin{equation}\label{ie:delta1}
2\pi n_k\widetilde{n}_k\chi(t-1)\delta_{\rho_t}<e^{\beta_0t},
\end{equation}
\begin{equation}\label{ie:delta2}
2\pi n_k\delta_{\rho_t}<\beta_3,
\end{equation}
and
\begin{equation}\label{ie:delta3}
\bigg(\frac{\epsilon}{2-\epsilon}e^t\bigg)^{2\pi
n_k\delta_{\rho_t}}>\frac{1}{1+\epsilon}.
\end{equation}

Since $\sigma_{t}\geq t-1$ on
$X_k$ and $\chi$ is decreasing, we have
$\chi(\sigma_{t})\leq \chi(t-1)$
on $X_k$. Then it follows from $(\ref{ie:twisted cuvature})$ and
$(\ref{ie:delta1})$ that
\[\Theta_{\rho_t,t}
\big|_{X_k\setminus(\Sigma_0\cup \Sigma_{\rho_t})}\geq
\frac{e^t}{ e^\psi }\sqrt{-1} \nu_t\wedge
\bar\nu_t-e^{\beta_0t}\omega.\] Hence
\begin{equation}\label{ie:B-estimate}
\mathrm{B}_{\rho_t,t}
+e^{\beta_0t}\mathrm{I}
\geq\bigg[\frac{e^t}
{ e^\psi }\sqrt{-1}\nu_t\wedge
\bar\nu_t,\,\Lambda \bigg]=\frac{e^t}
{ e^\psi }\mathrm{T}_{\bar\nu_t}
\mathrm{T}_{\bar\nu_t}^*\geq0
\end{equation}
on $X_k\setminus(\Sigma_0\cup \Sigma_{\rho_t})$ as an operator on
$(n,1)$-forms, where $\mathrm{T}_{\bar\nu_t}$ denotes the
operator $\bar\nu_t\wedge\bullet$ and
$\mathrm{T}_{\bar\nu_t}^*$ is its Hilbert adjoint
operator.

\textbf{Step 4: construction of suitably truncated forms and solving
$\bar\partial$ globally with $L^2$ estimates.}

In this step and Step 5, we will denote $\mathrm{B}_{\rho_t,t}$ and $h_{\rho_t,t}$ simply by $\mathrm{B}_t$ and $h_t$ respectively.

Let $\epsilon\in(0,\frac{1}{2})$ be as in Step 1. It is easy to construct a smooth function
$\theta:\mathbb{R}\longrightarrow[0,1]$ such that $\theta=0$ on
$(-\infty,\,\frac{\epsilon}{2}]$, $\theta=1$ on $[1-\frac{\epsilon}{2},\,+\infty)$
and $|\theta'|\leq \frac{1+\epsilon}{1-\epsilon}$ on $\mathbb{R}$.

Define $g_t=\mathrm{D}''
\big(\theta(\frac{ e^t}{ e^\psi + e^t})
\tilde{f}_t\big)$, where $\tilde{f}_t$ is
constructed in Step 1. Then $\mathrm{D}''g_t=0$ and
\begin{eqnarray*}
g_t&=&-\theta'\big(\frac{ e^t}
{ e^\psi + e^t}\big) \frac{ e^{\psi+t}}
{( e^\psi + e^t)^2}\bar\partial\psi\wedge
\tilde{f}_t+\theta\big(\frac{ e^t}
{ e^\psi + e^t}\big)\mathrm{D}''\tilde{f}_t\\
&=&g_{1,t}+g_{2,t},
\end{eqnarray*}
where $g_{1,t}$ denotes $-\bar\nu_t
\wedge\theta'(\frac{ e^t} { e^\psi + e^t})
\frac{ e^t} { e^\psi + e^t}\tilde{f}_t$ and
$g_{2,t}$ denotes $\theta(\frac{ e^t}
{ e^\psi + e^t})\mathrm{D}''\tilde{f}_t$. Then
\[\mathrm{supp}\,g_{1,t}\subset\{t-c_1<\psi<t+c_2\}\]
and
\[\mathrm{supp}\,g_{2,t}\subset\{\psi<t+c_2\},\]
where $c_1$ and $c_2$ are defined as in Step 1.

It follows from $(\ref{ie:inner product})$ and $(\ref{ie:B-estimate})$ that
\begin{eqnarray}
& &\langle(\mathrm{B}_t+2e^{\beta_0t} \mathrm{I})^{-1}
g_t,g_t
\rangle_{\omega,h_t}\big|_{X_k\setminus(\Sigma_0\cup \Sigma_{\rho_t})}\label{ie:estimate with delta}\\
&\leq&(1+\epsilon) \langle(\mathrm{B}_t+2e^{\beta_0t}
\mathrm{I})^{-1} g_{1,t},g_{1,t}
\rangle_{\omega,h_t} +\frac{1+\epsilon}{\epsilon}
\langle(\mathrm{B}_t+2e^{\beta_0t}
\mathrm{I})^{-1} g_{2,t},g_{2,t}
\rangle_{\omega,h_t}\nonumber\\
&\leq&(1+\epsilon) \langle(\mathrm{B}_t+e^{\beta_0t}
\mathrm{I})^{-1} g_{1,t},g_{1,t}
\rangle_{\omega,h_t} +\frac{1+\epsilon}{\epsilon}
\langle\frac{1}{e^{\beta_0t}}
g_{2,t},g_{2,t}
\rangle_{\omega,h_t}.\nonumber
\end{eqnarray}

By $(\ref{ie:B-estimate})$, we have
\[
\langle(\mathrm{B}_t+e^{\beta_0t} \mathrm{I})^{-1}
g_{1,t},g_{1,t}
\rangle_{\omega,h_t}\big|_{X_k\setminus(\Sigma_0\cup \Sigma_{\rho_t})}\leq\frac{ e^\psi }{e^t}\bigg|
\theta'\big(\frac{ e^t} { e^\psi + e^t}\big)
\frac{ e^t}
{ e^\psi + e^t}\tilde{f}_t\bigg|^2_{\omega,h_t}.
\]
Then $\zeta>0$ implies that
\begin{eqnarray*}
I_{1,t}
&:=&\int_{X_k\setminus(\Sigma_0\cup \Sigma_{\rho_t})}\langle
(\mathrm{B}_t+e^{\beta_0t} \mathrm{I})^{-1}
g_{1,t},g_{1,t}
\rangle_{\omega,h_t}dV_{X,\omega}\\
&\leq& \frac{(1+\epsilon)^2}{(1-\epsilon)^2}
\int_{X_k\cap\{t-c_1<\psi<t+c_2
\}}
\frac{ e^t|\tilde{f}_t|^2_{\omega,h_0}
e^{-\widetilde{\phi}_{\rho_t}\circ\mu^{-1} }
}{( e^\psi + e^t)^2e^{2\pi n_k\delta_{\rho_t}\psi}}dV_{X,\omega}.
\end{eqnarray*}

Since $\widetilde{\phi}_{\rho_t}\circ\mu^{-1} \geq\phi$ on
$X_k\setminus\Sigma_0$, it follows from $(\ref{ie:delta3})$ that
\begin{eqnarray*}
I_{1,t}
&\leq& \frac{(1+\epsilon)^2}{(1-\epsilon)^2}
\int_{X_k\cap\{t-c_1<\psi<t+c_2\}}
\frac{ e^t|\tilde{f}_t|^2_{\omega,h_0} e^{-\phi} dV_{X,\omega}}
{( e^\psi + e^t)^2
\big(\frac{\epsilon}{2-\epsilon}e^t\big)^{2\pi n_k\delta_{\rho_t}}}\\
&\leq&\frac{(1+\epsilon)^3}{(1-\epsilon)^2}
\int_{X_k\cap\{t-c_1<\psi<t+c_2\}}
\frac{ e^t|\tilde{f}_t|^2_{\omega,h_0}e^{-\phi}
dV_{X,\omega}}{( e^\psi + e^t)^2}.
\end{eqnarray*}
Since
\begin{equation*}
|\tilde{f}_t|^2_{\omega,h_0}\big|_{U}
=|\sum\limits_{i=1}^N\sqrt{\xi_i}\cdot\sqrt{\xi_i}\tilde{f}_{i,t}|^2_{\omega,h_0}
\leq(\sum\limits_{i=1}^N\xi_i)(\sum\limits_{i=1}^N\xi_i|\tilde{f}_{i,t}|^2_{\omega,h_0})
=\sum\limits_{i=1}^N\xi_i|\tilde{f}_{i,t}|^2_{\omega,h_0}
\end{equation*}
by the Cauchy-Schwarz inequality, we have
\begin{equation*}
I_{1,t}\leq\frac{ (1+\epsilon)^3}{(1-\epsilon)^2}\sum\limits_{i=1}^N
\int_{X_k\cap\{t-c_1<\psi<t+c_2\}}
\frac{ e^t\xi_i|\tilde{f}_{i,t}|^2_{\omega,h_0}
e^{-\phi}dV_{X,\omega}}{( e^\psi + e^t)^2}.
\end{equation*}
Then it follows from
$(\ref{ie:family-integral-limit})$ that
\begin{eqnarray*}
\varlimsup\limits_{ t\rightarrow-\infty } I_{1,t}
&\leq&\sum_{i=1}^N\varlimsup\limits_{ t\rightarrow-\infty }\bigg(\frac{ (1+\epsilon)^3}{(1-\epsilon)^2}
\int_{X_k\cap\{t-c_1<\psi<t+c_2\}}
\frac{ e^t\xi_i|\tilde{f}_{i,t}|^2_{\omega,h_0}
e^{-\phi}dV_{X,\omega}}{( e^\psi + e^t)^2}\bigg)\\
&\leq&\sum_{i=1}^N\frac{(1+\epsilon)^3}{(1-\epsilon)^2}\int_{U_i\cap Y^0}\xi_i|f|^2_{\omega,h_0}e^{-\phi}dV_{X,\omega}[\psi]\\
&\leq&\frac{(1+\epsilon)^3}{(1-\epsilon)^2}\int_{Y^0}|f|^2_{\omega,h_0}e^{-\phi}dV_{X,\omega}[\psi].
\end{eqnarray*}
Then
\begin{equation}\label{ie:I1}
I_{1,t}\leq \frac{(1+\epsilon)^4}{(1-\epsilon)^2}
\int_{Y^0}|f|^2_{\omega,h_0}e^{-\phi}dV_{X,\omega}[\psi]
\end{equation}
when $t$ is small enough.

Since $\zeta(\sigma_t)>0$ and
$\widetilde{\phi}_{\rho_t}\circ\mu^{-1} \geq\phi$ on
$X_k\setminus\Sigma_0$, by $(\ref{ie:delta2})$, we have
\begin{eqnarray*}
I_{2,t}&:=&\int_{X_k\setminus(\Sigma_0\cup \Sigma_{\rho_t})}\langle\frac{1}
{e^{\beta_0t}}
g_{2,t},g_{2,t}
\rangle_{\omega,h_t}dV_{X,\omega} \\
&\leq&\frac{1}{e^{\beta_0t}}\int_{X_k\cap\{\psi<t+c_2\}}
\frac{|\mathrm{D}''\tilde{f}_t|^2_{\omega,h_0}
e^{-\widetilde{\phi}_{\rho_t}\circ\mu^{-1}} }{e^{(1+2\pi n_k\delta_{\rho_t})\psi}} dV_{X,\omega}
 \\
&\leq&\frac{1}{e^{\beta_0t}} \int_{X_k\cap\{\psi<t+c_2\}}
\frac{|\mathrm{D}''\tilde{f}_t|^2_{\omega,h_0} e^{-\phi}}
{e^{(1+\beta_3)\psi}} dV_{X,\omega}.
\end{eqnarray*}
Then it follows from $(\ref{e:dbar of smooth extension})$ and the
Cauchy-Schwarz inequality that
$I_{2,t}$ is bounded by the sum of the terms
\[\frac{C_8}{e^{\beta_0t}}
\int_{U_i\cap U_j\cap\{\psi<t+c_2\}}
\frac{|\tilde{f}_{i,t}-\tilde{f}_{j,t}|^2_{\omega,h_0}e^{-\phi}}
{e^{(1+\beta_3)\psi}}dV_{X,\omega} \quad(1\leq i,j\leq N),\]
where $C_8$ is some positive number independent of $t$.

By the definition of $R_1$ (see Part II in Step 1), $(\ref{ie:special-extension--1})$ implies that for $i=1,\cdots,N$,
\begin{equation}\label{ie:special-extension-step4-1}
\int_{\Omega_i\cap \{\psi<t+c_2\}}
\frac{|\tilde{f}_{i,t}|^2_{\omega,h_0}e^{-(1+\beta)
\phi}}{e^\psi R_0(\psi)}dV_{X,\omega}\leq C_9
\end{equation}
for some positive number $C_9$ independent of
$t$ when $t$ is small enough. Then by the H\"{o}lder inequality, we get
\begin{eqnarray*}
& &\int_{U_i\cap U_j\cap\{\psi<t+c_2\}}
\frac{|\tilde{f}_{i,t}-\tilde{f}_{j,t}|^2_{\omega,h_0}e^{-\phi}}
{e^{(1+\beta_3)\psi}}dV_{X,\omega}\\
&\leq&\bigg(\int_{U_i\cap
U_j\cap\{\psi<t+c_2\}}
\frac{|\tilde{f}_{i,t}-\tilde{f}_{j,t}|^2_{\omega,h_0}e^{-(1+\beta)
\phi}}{e^\psi R_0(\psi)}dV_{X,\omega}\bigg)^{
\frac{1}{1+\beta}}\\
& &\times\bigg(\int_{U_i\cap
U_j\cap\{\psi<t+c_2\}}
\frac{|\tilde{f}_{i,t}-\tilde{f}_{j,t}|^2_{\omega,h_0}\big(
R_0(\psi)\big)^{\frac{1}{\beta}}}
{e^{(1+\beta_3\cdot\frac{1+\beta}{\beta})\psi}}
dV_{X,\omega}\bigg)^{\frac{\beta}{1+\beta}}\\
&\leq&C_{10}\bigg(\int_{U_i\cap
U_j\cap\{\psi<t+c_2\}}
\frac{|\tilde{f}_{i,t}-\tilde{f}_{j,t}|^2_{\omega,h_0}}
{e^{(1+\beta_3\cdot\frac{1+\beta}{\beta}+\beta_2\cdot\frac{1}{\beta})\psi}}
dV_{X,\omega}\bigg)^{\frac{\beta}{1+\beta}}
\end{eqnarray*}
when $t$ is small enough, where $C_{10}$ is a positive number
independent of $t$.

We will estimate the last integral above by estimating its pull-back under $\mu$. We cover $\mu^{-1}(U_i\cap U_j)\cap\{\psi\circ\mu<t+c_2\}$ by a finite number of coordinate balls such as $W$ in Step 1 in the proof of Proposition \ref{p:convergence-integral-psh}. It follows from $(\ref{ie:estimate of smooth extension1})$ and $(\ref{ie:estimate of smooth extension2})$ that for each $W$,
\begin{equation*}
\int_{W_{i,j,t}}
\frac{|\tilde{f}_{i,t}\circ\mu-\tilde{f}_{j,t}\circ\mu|^2_{\omega,h_0}|J_\mu|^2}
{e^{(1+\beta_3\cdot\frac{1+\beta}{\beta}+\beta_2\cdot\frac{1}{\beta})\psi\circ\mu}}
d\lambda(w)\leq C_{11}\int_{W_{i,j,t}}
\frac{1}
{\prod\limits_{p=1}^n|w_p|^{2\beta_{5,p}}}
d\lambda(w),
\end{equation*}
where
\[\beta_{5,p}:=\beta_4ca_p+(ca_p-b_p)-\lfloor ca_p-b_p\rfloor_+,\]
\[\beta_4:=\beta_3\cdot\frac{1+\beta}{\beta}+\beta_2\cdot\frac{1}{\beta}+\beta_1,\]
and $C_{11}$ is a positive number independent of $t$.

Since
\[\big(W\cap\{\psi\circ\mu<t+c_2\}\big)\subset \overset{n}{\underset{p=1}\cup}\big(\big\{|w_p|<e^{\frac{t+c_2-m}{2c|a|}}\big\}\cap W\big),\]
where $m:=\inf\limits_W\widetilde{u}(w)$, we obtain
\begin{eqnarray*}
\int_{W_{i,j,t}}
\frac{1}
{\prod\limits_{p=1}^n|w_p|^{2\beta_{5,p}}}
d\lambda(w)
&\leq&\sum\limits_{p=1}^n\int_{\big\{|w_p|<e^{\frac{t+c_2-m}{2c|a|}}\big\}\cap W}
\frac{1}
{\prod\limits_{p=1}^n|w_p|^{2\beta_{5,p}}}
d\lambda(w)\\
&\leq&C_{12}\sum\limits_{p=1}^ne^{\frac{1-\beta_{5,p}}{c|a|}t}
\end{eqnarray*}
when $\max\limits_{1\leq p\leq n}\beta_{5,p}<1$, where $C_{12}$ is a positive number independent of $t$.

Let $\beta_1$ be a positive number such that
\begin{equation}\label{ie:beta1-estimate}
\beta_1<\min\limits_{\{p:\,a_p\neq0\}}\frac{1-(ca_p-b_p)+\lfloor ca_p-b_p\rfloor_+}{3ca_p}.
\end{equation}

Take $\beta_2=\beta_1\beta$, $\beta_3=\frac{\beta_1\beta}{1+\beta}$. Then $\beta_4=3\beta_1$ and $\max\limits_{1\leq p\leq n}\beta_{5,p}<1$.

Let $\beta_0$ be a positive number such that
\[\beta_0<\min\limits_{1\leq p\leq n}\frac{\beta(1-\beta_{5,p})}{2(1+\beta)c|a|}\]
for every $W$. Then we have
\begin{equation}\label{ie:I2}
I_{2,t}\leq C_{13}\cdot e^{\beta_0t},
\end{equation}
where $C_{13}$ is a positive number independent of $t$.

Therefore, it follows from $(\ref{ie:estimate with delta})$, $(\ref{ie:I1})$ and $(\ref{ie:I2})$ that
\begin{equation*}
\int_{X_k\setminus(\Sigma_0\cup \Sigma_{\rho_t})}\langle(\mathrm{B}_t
+2e^{\beta_0t} \mathrm{I})^{-1}
g_t,g_t
\rangle_{\omega,h_t}dV_{X,\omega}\leq(1+\epsilon) I_{1,t}
+\frac{1+\epsilon}{\epsilon} I_{2,t}\leq C(t),
\end{equation*}
where
\[C(t):=\frac{(1+\epsilon)^5}{(1-\epsilon)^2}
\int_{Y^0}|f|^2_{\omega,h_0}e^{-\phi}dV_{X,\omega}[\psi]
+\frac{1+\epsilon}{\epsilon}C_{13}\cdot e^{\beta_0t}.\]

Then by Lemma \ref{l:Demailly-non complete metric}, there exists
$u_{k,\epsilon,t}\in L^2(X_k\setminus(\Sigma_0\cup \Sigma_{\rho_t}),\,K_X\otimes L,\,h_t)$ and
$v_{k,\epsilon,t}\in L^2(X_k\setminus(\Sigma_0\cup \Sigma_{\rho_t}),\,\wedge^{n,1}T_X^*\otimes L,\,h_t)$ such that
\begin{equation}\label{e:dbar}
\mathrm{D}''u_{k,\epsilon,t}
+\sqrt{2e^{\beta_0t}}v_{k,\epsilon,t} =g_t
\end{equation}
on $X_k\setminus(\Sigma_0\cup \Sigma_{\rho_t})$ and
\begin{eqnarray}
& &\int_{X_k\setminus(\Sigma_0\cup \Sigma_{\rho_t})}\frac{|u_{k,\epsilon,t}|^2
_{\omega,h_0}e^{-\widetilde{\phi}_{\rho_t} \circ\mu^{-1}-(1+2\pi
n_k\delta_{\rho_t}) \psi-\zeta(\sigma_t)}}
{\tau_t+A_t}dV_{X,\omega}\label{ie:estimate1}\\
& &+\int_{X_k\setminus(\Sigma_0\cup \Sigma_{\rho_t})}|v_{k,\epsilon,t}|^2_{\omega,h_0}
e^{-\widetilde{\phi}_{\rho_t} \circ\mu^{-1}-(1+2\pi
n_k\delta_{\rho_t})
\psi-\zeta(\sigma_t)}dV_{X,\omega}\nonumber\\
&\leq&C(t).\nonumber
\end{eqnarray}

Since $\{\widetilde{\phi}_{\rho_t}\circ\mu^{-1}\}$ are uniformly bounded above on
$X_k\setminus\Sigma_0$ with respect to $t$ as obtained in Step
2, we have
\begin{equation}\label{ie:varphi estimate}
e^{-\widetilde{\phi}_{\rho_t}\circ\mu^{-1}} \geq C_{14}
\end{equation}
on $X_k\setminus\Sigma_0$ for any $t$, where
$C_{14}$ is a positive number independent of $t$. Since
$t-\epsilon\leq\sigma_t\leq \alpha-\frac{\epsilon}{2}$ on
$\overline{X_k}$ and $\psi$ is upper semicontinuous on $X$, we
have that $\psi$, $\zeta(\sigma_t)$ and
$\tau_t+A_t$ are all bounded above on
$\overline{X_k}$ for each fixed $t$. Then it follows from
$(\ref{ie:estimate1})$ that $u_{k,\epsilon,t}\in
L^2$ and $v_{k,\epsilon,t}\in L^2$. Hence it follows from
$(\ref{e:dbar})$ and Lemma \ref{l:extension} that
\begin{equation}\label{e:dbar on X_k}
\mathrm{D}''u_{k,\epsilon,t}
+\sqrt{2e^{\beta_0t}}v_{k,\epsilon,t} =\mathrm{D}''
\bigg(\theta\big(\frac{ e^t} { e^\psi + e^t}\big)
\tilde{f}_t\bigg)
\end{equation}
holds on $X_k$. Furthermore, $(\ref{ie:estimate1})$ and $(\ref{5
ode2})$ imply that
\begin{eqnarray}
& &\int_{X_k}\frac{|u_{k,\epsilon,t}|^2
_{\omega,h_0}e^{-\widetilde{\phi}_{\rho_t} \circ\mu^{-1}}}
{\big(\frac{\alpha_1}{R(\alpha_0)}+C_R\big)e^\psi R(\sigma_t)}dV_{X,\omega}+\int_{X_k}|v_{k,\epsilon,t}|^2_{\omega,h_0}
e^{-\widetilde{\phi}_{\rho_t} \circ\mu^{-1}-
\psi-\zeta(\sigma_t)}dV_{X,\omega}\label{ie:estimate on X_k}\\
&\leq&e^{2\pi n_k\delta_{\rho_t} M_\psi}C(t),\nonumber
\end{eqnarray}
where $ M_\psi :=\sup\limits_{X_k}\psi$.

Define $F_{k,\epsilon,t}=-u_{k,\epsilon,t}
+\theta(\frac{ e^t}{ e^\psi + e^t})
\tilde{f}_t$. Then $(\ref{e:dbar on X_k})$ implies that
$\mathrm{D}''F_{k,\epsilon,t}
=\sqrt{2e^{\beta_0t}}v_{k,\epsilon,t}$ on $X_k$.
Since
$\widetilde{\phi}_{\rho_t}\circ\mu^{-1}\geq \phi$ on
$X_k\setminus\Sigma_0$, it follows from $(\ref{ie:inner product})$ and
$(\ref{ie:estimate on X_k})$ that
\begin{eqnarray}
& &\int_{X_k}\frac{|F_{k,\epsilon,t}|^2_{\omega,h_0}
e^{-\widetilde{\phi}_{\rho_t}\circ\mu^{-1}}}
{e^\psi \max\{R( \psi-\epsilon),R(\sigma_t)\}}dV_{X,\omega}\label{ie:estimate F0}\\
&\leq&(1+\epsilon)\int_{X_k}\frac{|u_{k,\epsilon,t}|^2_{\omega,h_0}
e^{-\widetilde{\phi}_{\rho_t}\circ\mu^{-1}}}{e^\psi
R(\sigma_t)}
dV_{X,\omega}\nonumber\\
& &+\frac{1+\epsilon}{\epsilon}
\int_{X_k}\frac{\big|\theta\big(\frac{ e^t}
{ e^\psi + e^t}\big) \tilde{f}_t\big|^2_{\omega,h_0}
e^{-\widetilde{\phi}_{\rho_t}\circ\mu^{-1}}}{e^\psi
R( \psi-\epsilon)}
dV_{X,\omega}\nonumber\\
&\leq&(1+\epsilon)e^{2\pi
n_k\delta_{\rho_t} M_\psi }\bigg(\frac{\alpha_1}{R(\alpha_0)}+C_R\bigg)C(t)
+\widetilde{C}(t)\nonumber
\end{eqnarray}
when $t$ is small enough, where
\[\widetilde{C}(t):=\frac{1+\epsilon}{\epsilon} \int_{X_k\cap\{\psi<t+c_2\}}
\frac{|\tilde{f}_t|^2_{\omega,h_0} e^{-\phi}}{e^\psi
R( \psi-\epsilon)}dV_{X,\omega}.\]

Now we want to prove
\begin{equation}\label{e:limit-second-C(t)}
\lim\limits_{ t\rightarrow-\infty }
\widetilde{C}(t)=0.
\end{equation}

As in $(\ref{ie:special-extension-step4-1})$, we can obtain from $(\ref{ie:special-extension--1})$ that for $i=1,\cdots,N$,
\begin{equation*}
\int_{\Omega_i\cap \{\psi<t+c_2\}}
\frac{|\tilde{f}_{i,t}|^2_{\omega,h_0}e^{-(1+\beta)
\phi}}{e^\psi R(\psi-\epsilon )}dV_{X,\omega}\leq
C_{15}
\end{equation*}
for some positive number $C_{15}$ independent of
$t$ when $t$ is small enough. Then by the H\"{o}lder
inequality, we have that
\begin{eqnarray*}
& &\int_{U_i\cap X_k\cap\{\psi<t+c_2\}}
\frac{|\tilde{f}_{i,t}|^2_{\omega,h_0} e^{-\phi}}{e^\psi
R( \psi-\epsilon)}dV_{X,\omega}\\
&\leq&\bigg(\int_{U_i\cap \{\psi<t+c_2\}}
\frac{|\tilde{f}_{i,t}|^2_{\omega,h_0}e^{-(1+\beta)
\phi}}{e^\psi R(\psi-\epsilon )}dV_{X,\omega}\bigg)^{\frac{1}{1+\beta}}\\
& &\times\bigg(\int_{U_i\cap
\{\psi<t+c_2\}}
\frac{|\tilde{f}_{i,t}|^2_{\omega,h_0}}
{e^\psi R(\psi-\epsilon )}dV_{X,\omega}\bigg)^{\frac{\beta}{1+\beta}}\\
&\leq&C_{15}^{\frac{1}{1+\beta}}\bigg(\int_{U_i\cap
\{\psi<t+c_2\}}
\frac{|\tilde{f}_{i,t}|^2_{\omega,h_0}}
{e^\psi R(\psi-\epsilon )}dV_{X,\omega}\bigg)^{\frac{\beta}{1+\beta}}
\end{eqnarray*}
when $t$ is small enough.

We cover $\mu^{-1}(U_i)\cap\{\psi\circ\mu<t+c_2\}$ by a finite number of coordinate balls such as $W$ in Step 1 in the proof of Proposition \ref{p:convergence-integral-psh}. Then, in order to prove $\lim\limits_{ t\rightarrow-\infty }
\widetilde{C}(t)=0$, it suffices to prove
\[\lim\limits_{t\rightarrow-\infty}\int_{W_{i,t}}
\frac{|\tilde{f}_{i,t}\circ\mu|^2_{\omega,h_0}|J_\mu|^2}
{e^{\psi\circ\mu} R(\psi\circ\mu-\epsilon )}d\lambda(w)=0,\]
where
\[W_{i,t}:=W\cap\mu^{-1}(U_i)\cap
\{\psi\circ\mu<t+c_2\}.\]
Then by $(\ref{ie:case1-ft})$, $(\ref{ie:f-estimate})$, $(\ref{ie:case2-ft-1})$ and $(\ref{ie:case2-ft-2})$, it suffices to prove
\begin{equation}\label{e:lim-small-term1}
\lim\limits_{t\rightarrow-\infty}\int_{W_{i,t}}\frac{d\lambda(w)}{R(\psi\circ\mu-\epsilon )|w_{p_0}|^2\prod\limits_{1\leq p\leq n,p\neq p_0}|w_p|^{2(ca_p-b_p)-2\lfloor ca_p-b_p\rfloor_+}}=0
\end{equation}
in Case $(A)$ and
\begin{equation}\label{e:lim-small-term2}
\lim\limits_{t\rightarrow-\infty}\int_{W_{i,t}}\frac{d\lambda(w)}{R(\psi\circ\mu-\epsilon )\prod\limits_{p=1}^n|w_p|^{2\beta_1ca_p+2(ca_p-b_p)-2\lfloor ca_p-b_p\rfloor_+}}=0
\end{equation}
in Case $(A)$ and Case $(B)$.

Applying Fubini's theorem with respect to $(w',w_{p_0})$ and then using change of variables, we can obtain that
\begin{eqnarray*}
& &\lim\limits_{t\rightarrow-\infty}\int_{W_{i,t}}\frac{d\lambda(w)}{R(\psi\circ\mu-\epsilon )|w_{p_0}|^2\prod\limits_{1\leq p\leq n,p\neq p_0}|w_p|^{2(ca_p-b_p)-2\lfloor ca_p-b_p\rfloor_+}}\\
&\leq&C_{16}\lim\limits_{t\rightarrow-\infty}\int_{-\infty}^{t+c_2-m}\frac{ds}{R(s+M-\epsilon)}\\
&=&0,
\end{eqnarray*}
where $M:=\sup\limits_{W}\widetilde{u}(w)$, $m:=\inf\limits_{W}\widetilde{u}(w)$ and $C_{16}$ is a positive number independent of $t$. Hence we get $(\ref{e:lim-small-term1})$.

Similarly, it is easy to see that $(\ref{ie:beta1-estimate})$ implies that $(\ref{e:lim-small-term2})$.

Therefore, we obtain $(\ref{e:limit-second-C(t)})$.

Let $\widehat{\alpha}_k:=\sup\limits_{X_k}\alpha$. Then
\[e^\psi\max\{R(\psi-\epsilon),R(\sigma_t)\}\leq  e^{\epsilon}
\sup\limits_{t\leq\widehat{\alpha}_k}\big(e^tR(t)\big).\]
Hence it follows from
$(\ref{ie:varphi estimate})$ and $(\ref{ie:estimate F0})$ that
\begin{equation}\label{ie:F-no-weight-L2}
\int_{X_k}|F_{k,\epsilon,t}|^2_{\omega,h_0}dV_{X,\omega}\leq  C_{17}
\end{equation}
for some positive number $ C_{17}$ independent of
$t$ when $t$ is small enough.

Since the positive continuous function $R$ is decreasing near $-\infty$, it is easy to see that $\max\{R(\psi-\epsilon),R(\sigma_t)\}$ is equal to $R(\psi-\epsilon)$ near $\{\psi=-\infty\}$ and converges uniformly to $R(\psi-\epsilon)$ on $\overline{X_k}$ as $t\rightarrow-\infty$.

Since $\widetilde{\phi}_{\rho_t}\circ\mu^{-1}$ is
increasing with respect to $t$ and converges to $\phi$ on
$X_k\setminus\Sigma_0$ as $ t\rightarrow-\infty $, by extracting weak
limits of $\{F_{k,\epsilon,t}\}$ as
$ t\rightarrow-\infty $, we get from $(\ref{ie:F-no-weight-L2})$
and $(\ref{ie:estimate F0})$ a sequence
$\{t_j\}_{j=1}^{+\infty}$ and $F_{k,\epsilon}\in
L^2$ such that
$\lim\limits_{j\rightarrow+\infty}t_j=-\infty$,
$F_{k,\epsilon,t_j}\rightharpoonup F_{k,\epsilon}$ weakly in
$L^2$ as $j\rightarrow+\infty$ and
\begin{equation}\label{ie:estimate F1}
\int_{X_k}\frac{|F_{k,\epsilon}|^2_{\omega,h_0}e^{-\phi}}{e^\psi
R( \psi-\epsilon)}dV_{X,\omega} \leq
\frac{(1+\epsilon)^6}{(1-\epsilon)^2}\bigg(\frac{\alpha_1}{R(\alpha_0)}+C_R\bigg)
\int_{Y^0}|f|^2_{\omega,h_0}e^{-\phi}dV_{X,\omega}[\psi].
\end{equation}

Since $\sigma_t\leq\alpha-\frac{\epsilon}{2}$ on $X_k$, $\widehat{\alpha}_k:=\sup\limits_{X_k}\alpha$ and $\zeta$ is
increasing, we get
\begin{equation}\label{ie:zeta estimate}
e^{-\zeta(\sigma_t)}
\geq e^{-\zeta(\widehat{\alpha}_k-\frac{\epsilon}{2})}
\end{equation}
on $X_k$. Then $(\ref{ie:estimate on X_k})$, $(\ref{ie:varphi
estimate})$ and $(\ref{ie:zeta estimate})$ imply that
\begin{equation*}
\int_{X_k} |v_{k,\epsilon,t}|^2_{\omega,h_0}dV_{X,\omega}\leq
e^{\zeta(\widehat{\alpha}_k-\frac{\epsilon}{2})+(1+2\pi
n_k\delta_{\rho_t}) M_\psi }C^{-1}_{14}C(t).
\end{equation*}
Hence $\sqrt{2e^{\beta_0t_j}}v_{k,\epsilon,t_j}
\rightarrow0$ in $L^2$ as
$j\rightarrow+\infty$. Since $\mathrm{D}''F_{k,\epsilon,t}
=\sqrt{2e^{\beta_0t}}v_{k,\epsilon,t}$ on $X_k$, we
get $\mathrm{D}''F_{k,\epsilon}=0$ on $X_k$. Then $F_{k,\epsilon}$ is a
holomorphic section of $K_X\otimes L$ on $X_k$. In Step 5, we will
prove that $F_{k,\epsilon}=f$ on $X_k\cap Y^0$ by solving $\bar\partial$
locally.

\textbf{Step 5: solving $\bar\partial$ locally with $L^2$ estimates
and the end of the proof for the line bundle $L$.}

For any $x\in X_k\cap Y$, let $\Omega_x$ be as in Step 1. Let
\[\widehat{\Omega}_x\subset\subset (X_k\cap\Omega_x)\]
be a coordinate ball with center $x$. Since
the bundle $L$ is trivial on $\Omega_x$, $u_{k,\epsilon,t}$ and
$v_{k,\epsilon,t}$ can be regarded as forms on $\Omega_x$ with
values in $\mathbb{C}$ and the metric $h_0$ of $L$ on $\Omega_x$ can be
regarded as a positive smooth function.

It is easy to see that $C(t)\leq C_{18}$ for some positive
number $C_{18}$ independent of $t$ when $t$ is
small enough. Then it follows from $(\ref{ie:estimate on X_k})$,
$(\ref{ie:zeta estimate})$ and
$(\ref{ie:varphi estimate})$ that
\[\int_{\widehat{\Omega}_x}|v_{k,\epsilon,t}|^2 e^{-\psi}d\lambda \leq
C_{19}C_{18}\]
for some positive number $C_{19}$ independent of $t$ when
$t$ is small enough.

Since $\bar\partial v_{k,\epsilon,t}=0$ on $\widehat{\Omega}_x$ by
$(\ref{e:dbar on X_k})$, applying Lemma \ref{l:Hormander-dbar} to
the $(n,1)$-form
\[\sqrt{2e^{\beta_0t}}v_{k,\epsilon,t}\in
L^2_{(n,1)}(\widehat{\Omega}_x,\,e^{-\psi}),\] we get an $(n,0)$-form
$s_{k,\epsilon,t}\in L^2_{(n,0)}(\widehat{\Omega}_x,\,e^{-\psi})$ such
that \[\bar\partial s_{k,\epsilon,t}
=\sqrt{2e^{\beta_0t}}v_{k,\epsilon,t}\] on $\widehat{\Omega}_x$ and
\begin{equation}\label{ie:estimate v1}
\int_{\widehat{\Omega}_x}|s_{k,\epsilon,t}|^2
e^{-\psi}d\lambda \leq
C_{20}\int_{\widehat{\Omega}_x}|\sqrt{2e^{\beta_0t}} v_{k,\epsilon,t}|^2
e^{-\psi}d\lambda\leq 2C_{20}C_{19}C_{18}e^{\beta_0t}
\end{equation}
for some positive number $C_{20}$ independent of $t$. Hence
\begin{equation}\label{ie:estimate v2}
\int_{\widehat{\Omega}_x}|s_{k,\epsilon,t}|^2 d\lambda \leq
C_{21}e^{\beta_0t}
\end{equation}
for some positive number $C_{21}$ independent of $t$.

Now define $G_{k,\epsilon,t}=-u_{k,\epsilon,t}
-s_{k,\epsilon,t}
+\theta(\frac{ e^t}{ e^\psi + e^t})
\tilde{f}_t$ on $\widehat{\Omega}_x$. Then
$G_{k,\epsilon,t}=F_{k,\epsilon,t}
-s_{k,\epsilon,t}$ and $\bar\partial
G_{k,\epsilon,t}=0$. Hence $G_{k,\epsilon,t}$ is
holomorphic in $\widehat{\Omega}_x$. Therefore,
$u_{k,\epsilon,t}+s_{k,\epsilon,t}$ is smooth in
$\widehat{\Omega}_x$. Furthermore, we get from $(\ref{ie:F-no-weight-L2})$ and
$(\ref{ie:estimate v2})$ that
\begin{equation}\label{ie:G-estimate}
\int_{\widehat{\Omega}_x}|G_{k,\epsilon,t}|^2d\lambda
\leq2\int_{\widehat{\Omega}_x}|F_{k,\epsilon,t}|^2d\lambda
+2\int_{\widehat{\Omega}_x}|s_{k,\epsilon,t}|^2d\lambda \leq C_{22}
\end{equation}
for some positive number $C_{22}$ independent of $t$ when
$t$ is small enough.

We get
from $(\ref{ie:varphi estimate})$ and $(\ref{ie:estimate on X_k})$ that
\begin{equation*}
\int_{\widehat{\Omega}_x}\frac{|u_{k,\epsilon,t}|^2
e^{-\psi}}{R(\sigma_t)}d\lambda \leq
C_{23}C(t)\leq C_{23}C_{18}
\end{equation*}
for some positive number $C_{23}$ independent of $t$ when
$t$ is small enough. Since $R(\sigma_t)\leq
R(t-\epsilon)$ on $\widehat{\Omega}_x$ when $t$ is small enough, we have that
\begin{equation*}
\int_{\widehat{\Omega}_x}|u_{k,\epsilon,t}|^2 e^{-\psi}d\lambda \leq
C_{23}C_{18}R(t-\epsilon).
\end{equation*}

Therefore, combining the last inequality and $(\ref{ie:estimate
v1})$, we obtain that
\[\int_{\widehat{\Omega}_x}|u_{k,\epsilon,t}
+s_{k,\epsilon,t}|^2e^{-\psi}d\lambda \leq
2C_{23}C_{18}R(t-\epsilon) +4C_{20}C_{19}C_{18}e^{\beta_0t}.\]
Then the non-integrability of $e^{-\psi}$ along $\widehat{\Omega}_x\cap Y$ and
the smoothness of $u_{k,\epsilon,t}+s_{k,\epsilon,t}$
in $\widehat{\Omega}_x$ show that
$u_{k,\epsilon,t}+s_{k,\epsilon,t}=0$ on $\widehat{\Omega}_x\cap Y$ for any $t$.
Hence $G_{k,\epsilon,t}=f$ on $\widehat{\Omega}_x\cap Y^0$ for any $t$.

Since $s_{k,\epsilon,t_j}\rightarrow 0$ in $L^2_{(n,0)}(\widehat{\Omega}_x)$
by $(\ref{ie:estimate v2})$ and
$F_{k,\epsilon,t_j}\rightharpoonup F_{k,\epsilon}$ weakly in
$L^2_{(n,0)}(\widehat{\Omega}_x)$ as $j\rightarrow+\infty$,
$G_{k,\epsilon,t_j}\rightharpoonup F_{k,\epsilon}$ weakly in
$L^2_{(n,0)}(\widehat{\Omega}_x)$ as $j\rightarrow+\infty$. Hence it follows from
$(\ref{ie:G-estimate})$ and routine arguments with applying Montel's
theorem that a subsequence of
$\{G_{k,\epsilon,t_j}\}_{j=1}^{+\infty}$ converges to
$F_{k,\epsilon}$ uniformly on compact subsets of $\widehat{\Omega}_x$. Then
$F_{k,\epsilon}=f$ on $\widehat{\Omega}_x\cap Y^0$ and thereby on $X_k\cap Y^0$.

Since the positive continuous function $R$ is decreasing near $-\infty$, $e^tR(t)$ is bounded above near $-\infty$ and $\phi$ is locally bounded above, applying Montel's
theorem and extracting weak limits of $\{F_{k,\epsilon}\}_{k,\epsilon}$,
first as
$\epsilon\rightarrow 0$, and then as $k\rightarrow+\infty$, we get from
$(\ref{ie:estimate F1})$ a holomorphic section $F$ on $X$ with
values in $K_X\otimes L$ such that $F=f$ on $Y^0$ and
\begin{equation*}
\int_{X}\frac{|F|^2_{\omega,h}}{e^\psi
R( \psi)}dV_{X,\omega} \leq\bigg(\frac{\alpha_1}{R(\alpha_0)}+C_R\bigg)
\int_{Y^0}|f|^2_{\omega,h}dV_{X,\omega}[\psi].
\end{equation*}
Theorem \ref{t:Zhou-Zhu1} is thus proved for the line bundle $L$.

\textbf{Step 6: the proof for the vector bundle $E$.}

The proof for $E$ is similar but simpler. We only point out the main modifications by examining the proof for $L$.

In Step 1, we don't need to construct a family of special smooth extensions $\tilde{f}_t$ of $f$ since $h_E$ is smooth. Hence the strong openness property  and the key propositions are not needed. Delete Part II and Part III in Step 1 and replace the family of sections $\tilde{f}_{i,t}$ with a fixed local holomorphic extension $\tilde{f}_i$. Then $\tilde{f}_t$ becomes a fixed smooth extension $\tilde{f}=\sum\limits_{i=1}^N\xi_i\tilde{f}_i$. Then it is easy to see that $(\ref{ie:family-integral-limit})$, $(\ref{e:dbar of smooth extension})$, $(\ref{ie:estimate
of smooth extension1})$ and $(\ref{ie:estimate of smooth extension2})$ hold for $\tilde{f}_{i,t}=\tilde{f}_i$, $\tilde{f}_t=\tilde{f}$ and $\beta_1=0$.

Step 2 is not needed since $h_E$ is already smooth.

In Step 3, the negative term will not appear on the right hand side of $(\ref{ie:twisted cuvature})$ since $\delta_\rho=0$.

In Step 4, it is easy to prove the estimate $(\ref{ie:I1})$ for $I_{1,t}$ by the modified $(\ref{ie:family-integral-limit})$. It is also not hard to prove the estimate $(\ref{ie:I2})$ for $I_{2,t}$ by the modified $(\ref{e:dbar of smooth extension})$, $(\ref{ie:estimate
of smooth extension1})$ and $(\ref{ie:estimate of smooth extension2})$. $(\ref{e:limit-second-C(t)})$ can be easily obtained since $h_E$ is smooth.

Step 5 for $E$ is almost the same and Theorem \ref{t:Zhou-Zhu1} is thus proved for the vector bundle $E$.





\bibliographystyle{references}
\bibliography{xbib}

\end{document}